\documentclass[letterpaper,oneside,11pt,reqno]{amsart}

\usepackage[english]{babel} 
\usepackage[T1]{fontenc}
\usepackage[utf8]{inputenc}
\usepackage{lmodern} 
\usepackage[abbrev]{amsrefs}
\renewcommand{\MR}[1]{}
\usepackage{hyphenat}
\usepackage{dsfont}

\usepackage{amsmath, amssymb}
\usepackage{amsthm}
\usepackage{amsfonts}
\newtheorem{teo}{Theorem}[section]
\newtheorem{prop}[teo]{Proposition}
\newtheorem{lemma}[teo]{Lemma}
\newtheorem{re}[teo]{Remark}
\newtheorem{de}[teo]{Definition}
\newtheorem{cor}[teo]{Corollary}

\numberwithin{equation}{section}
\newcommand{\Real}{\mathbb R}

\usepackage{mathtools}
\mathtoolsset{showonlyrefs}
\usepackage{enumerate}

\usepackage{parskip}
\usepackage{xcolor}
\usepackage[normalem]{ulem}

\usepackage[plainpages=false,pdfpagelabels,pdfusetitle,pdfdisplaydoctitle, pdfstartpage=1, pdfstartview={{XYZ null null 1.00}}]{hyperref}
\usepackage[top=3cm, bottom=2cm, left=3cm, right=2cm]{geometry}
\author{\textsc{Luccas Campos} and \textsc{Ademir Pastor}}

\renewcommand{\Re}{\operatorname{Re}}
\renewcommand{\Im}{\operatorname{Im}}

\usepackage{etoolbox}
\makeatletter
\patchcmd{\l@section}{1.0em}{0.4em}{}{}
\makeatother

\begingroup 
\makeatletter
   \@for\theoremstyle:=theorem,de,re,pro,lem,coro,plain\do{%
     \expandafter\g@addto@macro\csname th@\theoremstyle\endcsname{%
        \addtolength\thm@preskip\parskip
     }%
   }
\endgroup

\let\oldphi\phi 
\let\phi\varphi 
\let\varphi\oldphi

\begin{document}
\scrollmode
\title{Threshold solutions for cubic Schr\"odinger systems}
\date{}




\begin{abstract}\noindent
We consider the following Scr\"odinger system
\begin{equation}
	\begin{cases}
	    i\partial_t u + \Delta u +(|u|^2+\beta |v|^2) u= 0,\\
	    i\partial_t v + \Delta v +(|v|^2+\beta |u|^2) v = 0,\\
	\end{cases}
\end{equation}
with initial data $(u_0,v_0) \in H^1(\Real ^3)\times H^1(\Real^3)$ at the so-called \textit{mass-energy threshold}, i.e., such that 
$M(u_0,v_0)E(u_0,v_0) = M(\phi,\psi)E(\phi,\psi)$, where $(\phi,\psi)$ is a ground state. For a suitable range of values of $\beta>0$, we show the existence of special solutions to this system, which converge to a standing wave solution in one time direction, and either blows up or scatters in the opposite direction. Moreover, we classify general solutions at the ground state, showing a rigidity result regarding the possible long-time behaviors that might occur. Our results do not rely on the uniqueness of the corresponding ground state: indeed, the main results hold even in the case where the Weinstein functional is known to have more than one optimizer. 

\end{abstract}
    \vspace{-.4cm}

    \maketitle
\tableofcontents




\section{Introduction}

We consider the following nonlinear Schrödinger (NLS) system  of equations 
\begin{equation}\label{sys_NLS}
	\begin{cases}
	    i\partial_t u + \Delta u +(|u|^2+\beta |v|^2) u= 0,\\
	    i\partial_t v + \Delta v +(|v|^2+\beta |u|^2) v = 0,\\
	\end{cases}
\end{equation}
where $u$ and $v$ are complex-valued functions on the variables $(x,t) \in \mathbb{R}^{3} \times \mathbb{R}$, and $\beta>0$ is a coupling constant. From the physical point of view, system \eqref{sys_NLS} appears, for example, in nonlinear optics when higher-order nonlinear effects are included such as stimulated Raman scattering (SRS) and
stimulated Brillouin scattering (SBS).
For instance, if the
peak power of the incident pulse is above a threshold level then both SRS and SBS
may transfer energy from the pulse to a new pulse. The two pulses interact with each other through the Raman or Brillouin gain. A similar situation occurs when two
or more pulses at different wavelengths are incident on the fiber (see \cite{agrawalbook}). The interested reader will find a wide discussion concerning nonlinear optics, for instance, in \cite{agrawalbook} or \cite{boydbook}.

The initial value problem associated with \eqref{sys_NLS} may be studied following the same strategies for the standard cubic NLS equation,
\begin{equation}\label{standardNLS}
    i\partial_t u+\Delta u +|u|^2u=0.
\end{equation}
In particular local well-posedness in the usual $L^2$-based Sobolev spaces is well understood. We are interested in considering \eqref{sys_NLS} mainly in the inhomogeneous Sobolev space $H^1(\Real^3)\times H^1(\Real^3)$. Therefore, given an initial data $(u_0,v_0)\in H^1(\Real^3)\times H^1(\Real^3)$, there exist $T_*=T_*(u_0,v_0)\in(0,\infty]$, $T^*=T^*(u_0,v_0)\in(0,\infty]$ and a unique solution of \eqref{sys_NLS} satisfying $(u,v)\in C((-T_*,T^*):H^1(\Real^3)\times H^1(\Real^3))$. In addition, the following blow-up alternative holds: if $T^*<\infty$ then
\[
\lim_{t\to T^*}\|(u(t),v(t))\|_{H^1\times H^1}=\infty.
\]
In this case we say that solution \textit{blows-up in finite positive time}. Similarly, if $T_*<\infty$ then
\[
\lim_{t\to -T_*}\|(u(t),v(t))\|_{H^1\times H^1}=\infty
\]
and we say that the solution \textit{blows-up in finite negative time}. We refer the reader to \cite{Cazenave03}*{Chapter 4} or \cite{LiPo15}*{Chapter 5} for the proofs in the case of the scalar Schr\"odinger equation; but a very similar analysis also establishes the results for system \eqref{sys_NLS}. If both $T_*=\infty$ and $T^*=\infty$ we say that the solution is \textit{global} and \eqref{sys_NLS} is said to be \textit{globally well-posed}.

Note that \eqref{sys_NLS} is invariant by scaling. Namely, if $(u,v)$ is a solution with initial data $(u_0,v_0)$, then \begin{equation}
    (u_{\delta}(x,t),v_{\delta}(x,t)) =  (\delta u(\delta x, \delta^2 t),\delta v(\delta x, \delta^2 t)), \qquad \delta > 0,
\end{equation}
is also a solution with initial data $(u_\delta(0),v_\delta(0))=(\delta u_0(\delta x),\delta v_0(\delta x))$. Computing the homogeneous $\dot{H}^s(\Real^3)$ norm yields
\begin{equation}
    \| (u_{\delta}(0),v_{\delta}(0))\|_{\dot{H}^s\times \dot{H}^s} = \delta^{s-
   1/2}\|(u_0,v_0)\|_{\dot{H}^s\times \dot{H}^s}.
\end{equation}
Hence, the scale-invariant Sobolev space is $\dot{H}^{\frac{1}{2}}(\Real^3)\times \dot{H}^{\frac{1}{2}}(\Real^3)$. Such a space is sometimes referred to as the critical Sobolev space. In particular we see that the critical regularity is below that of our target space $H^1(\Real^3)\times H^1(\Real^3)$. 

Consequently, in order to prove a global well-posedness result it suffices to guarantee, as usual, an a priori bound of the solution in hand. As for the scalar NLS equation, such a priori bound may be obtained with the help of the available conservation laws. Here we recall that  system \eqref{sys_NLS} conserves the \textit{energy}
\begin{equation}
    E(u(t),v(t)) := \frac{1}{2} \int \left[ |\nabla u(t)|^2 +|\nabla v(t)|^2\right] - 
    \frac{1}{4}\int \left[|u|^4+2\beta|uv|^2+|v|^4\right]
\end{equation}
the \textit{mass}
\begin{equation}
    M(u(t),v(t)) := \int\left[ |u(t)|^2 +  |v(t)|^2 \right]
\end{equation}
and the \textit{momentum}
\begin{equation}
    F(u(t),v(t)) := \Im\int \left[ \bar{u}(t)\nabla u(t)+\bar{v}(t)\nabla v(t)\right].
\end{equation}

The dicothomy  global well-posedness versus blow-up in finite time was studied in \cite{Pastor}. To give a precise description of the results, recall that a \textit{standing wave} for \eqref{sys_NLS} is a solution of the form
\begin{equation}\label{standing}
    {u(x,t) =  (e^{i t}\phi(x),e^{i t} \psi(x)) },
\end{equation}
where $\phi$ and $\psi$ are real-valued functions in $H^1(\Real^3)$. By substituting \eqref{standing} in \eqref{sys_NLS} we see that $\phi$ and $\psi$ must be solutions of the following nonlinear elliptic system
\begin{align}\label{sys_def_ground}
\begin{cases}
        \Delta \phi - \phi +(|\phi|^2+\beta |\psi|^2)\phi=0,\\
        \Delta \psi - \psi +(|\psi|^2+\beta |\phi|^2)\psi = 0.
\end{cases}
\end{align}
A (weak) solution of \eqref{sys_def_ground} is also called  a \textit{bound state}. It is easy to see that bound states are critical points of the functional 
$$
I(\varphi,\psi)=E(\varphi,\psi)+\frac{1}{2}M(\varphi,\psi).
$$
Among all bound states, the minimizers of $I$ play a fundamental role in the dynamics of \eqref{sys_NLS}. These particular solutions are called \textit{ground states}. We will denote the set of all ground states by $\mathcal{G}$.

 Take any $(\phi,\psi) \in \mathcal{G}$ and introduce the quantities 
\begin{equation}\label{MKdef}
    \mathcal{MK}(f,g):=\dfrac{M(f,g)K(f,g)}{M(\phi,\psi)K(\phi,\psi)} \quad \mbox{and} \quad \mathcal{ME}(f,g):=\dfrac{M(f,g)E(f,g)}{M(\phi,\psi)E(\phi,\psi)},
\end{equation}
where here and throughout the whole paper, 
$$
K(f,g)=\int \left[ |\nabla f|^2 +|\nabla g|^2\right].
$$
By using \eqref{sys_def_ground} and integration by parts one can see that $M(\phi,\psi) = M(\tilde{\phi},\tilde{\psi})$, $K(\phi,\psi) = K(\tilde{\phi},\tilde{\psi})$, and $E(\phi,\psi) = E(\tilde{\phi},\tilde{\psi})$ for all $(\phi,\psi), (\tilde{\phi},\tilde{\psi}) \in \mathcal{G}$. In particular the denominators in \eqref{MKdef} 
do not depend on the choice of the ground state.

The results established in \cite{Pastor} may be summarized as follows: assume $(u_0,v_0)\in H^1(\Real^3)\times H^1(\Real^3)$ and 
\begin{equation}\label{MEcond}
    \mathcal{ME}(u_0,v_0)<1.
\end{equation}
If $\mathcal{MK}(u_0,v_0)<1$ then the corresponding solution is global. Also, if $\mathcal{MK}(u_0,v_0)>1$ and $u_0,v_0$ are radial functions (or satisfy $|x|(u_0,v_0)\in L^2(\Real^3)\times L^2(\Real^3))$ then the corresponding solution blows-up in negative and positive finite time. These results were complemented in \cite{FP}, where the authors proved that in the case of a global solution it also scatters forward and backward in time. Recall we say that a global solution $(u(t),v(t))$ of \eqref{sys_NLS} scatters forward in time if there exist $u_0^+,v_0^+\in H^1(\Real^3)$ such that 
\begin{equation}\label{defscat}
\lim_{t\to+\infty}\|(u(t),v(t))-(e^{it\Delta}u_0^+,e^{it\Delta}v_0^+)\|_{H^1\times H^1}=0.
\end{equation}
Also, we say that  $(u(t),v(t))$ scatters backward in time if there exist $u_0^-,v_0^-\in H^1(\Real^3)$ such that 
\begin{equation}\label{defscat-}
\lim_{t\to-\infty}\|(u(t),v(t))-(e^{it\Delta}u_0^-,e^{it\Delta}v_0^-)\|_{H^1\times H^1}=0.
\end{equation}
Here, $e^{it\Delta}f$ stands for the free solution of the scalar Schr\"odinger equation with initial data $f$.

If $(\phi, \psi)$ solves \eqref{sys_def_ground}, then the {standing wave} given in \eqref{standing}
 neither blows up in finite time, nor scatters, in any time direction. Hence, it seems that the assumption \eqref{MEcond} is crucial to obtain the above mentioned results. In the present paper we are interested in the dynamics of \eqref{sys_NLS} exactly at the threshold, that is, in the case $\mathcal{ME}(u_0,v_0)=1$. 
 
 Before stating our results, let us discuss a little more concerning the ground states of \eqref{sys_def_ground}. It is immediate to check that the pair $((1+\beta)^{-1/2}\,\varphi,(1+\beta)^{-1/2}\,\varphi)$ solves system \eqref{sys_def_ground} if $\varphi$ is a (weak) solution of the scalar equation
\begin{equation}\label{scalar_ground}
    \Delta \varphi-\varphi+
    |\varphi|^2\varphi=0.
\end{equation}
Such equation is related to the scalar NLS equation \eqref{standardNLS} and it is well known (c.f. \cite{Strauss77}, \cite{BLP81}, \cite{Kwong89},\cite{TaoBook}*{Appendix B}) that it admits a unique radial positive $H^1$ solution (the ground state), which is smooth and decays exponentially. Besides the \textit{trivial solution} $(\phi,\psi) = (0,0)$, system \eqref{sys_def_ground} also exhibits \textit{semi-trivial} solutions of the form $(\varphi,0)$ and $(0,\varphi)$, where $\varphi$ solves the scalar equation \eqref{scalar_ground}. 

Therefore, differently from the scalar case, ground states are not necessarily unique up to symmetries. Indeed, the results given in \cite{Mande115}*{Lemma 1 and Lemma 2}, \cite{Maia}*{Corollary 2.6} and \cite{WY_Uniq}*{Theorem 1.3} (see also \cite{STH16} and \cite{Correia16}) show that if $\beta>1$, then there is a unique ground state (up to phase, scaling and translation), and it is given by the fully nontrivial solution $((1+\beta)^{-1/2}\,\varphi,(1+\beta)^{-1/2}\,\varphi)$, where $\varphi$ is the radial positive solution of \eqref{scalar_ground}. On the other hand, if the coupling is weak, i.e., if $0 < \beta < 1$, then the ground states are the semi-trivial solutions $(\varphi,0)$ and $(0,\varphi)$. In view of this, throughout the paper we will assume
\begin{equation}\label{betarange}
    \text{either } 0 < \beta < 1 \text{ or } \beta > 1.
\end{equation}

The case $\beta=1$ is degenerated, as there appear infinitely many radial non-negative solutions given by $(\lambda\, \varphi, \sqrt{1-\lambda^2}\, \varphi)$, for any $0 \leq \lambda \leq 1$.

Since there can be more than one ground state, the main challenge here, which has no analogue in the scalar case, is to understand the possible behaviors this situation leads to. The variational characterization of the ground states must take the lack of uniqueness into account, and the approach has to be changed, given the uniqueness assumptions were heavily used in several works (c.f. \cites{DM_Dyn,DR_Thre,higher_thre,CFR_thre}).




Our first result shows that the dynamics of \eqref{sys_NLS} at the threshold is even richer.

\begin{teo}[Special solutions]\label{sub_special}
Assume $\beta$ satisfies \eqref{betarange}. For each ground state $Q = (\phi,\psi) \in \mathcal{G}$ there exist two radial solutions $Q^+=(\phi^+,\psi^+)$ and $Q^-=(\phi^-,\psi^-)$ to \eqref{sys_NLS} in $H^1(\Real^3)\times H^1(\Real^3)$ such that
\begin{itemize}
\item[(i)] $M(Q^\pm) = M(Q)$, $E(Q^\pm) = E(Q)$ , $T^*(Q^\pm)=+\infty$ and there exist $C$, $e_0>0$ such that
\begin{equation}
\|Q^\pm(t) - e^{it}Q\|_{H^1 \times H^1} \leq C e^{-e_0t}\text{ for all } t \geq 0,
\end{equation} 
\item[(ii)] $K(Q^+(0)) >K(Q)$ and $T_*(Q^+)<+\infty$
\item[(iii)] $K(Q^-(0))< K(Q)$, $T_*(Q^-)=+\infty$ 
and $Q^-$  scatters backward in time.
\end{itemize}
\end{teo}

We then show that, besides the behavior already expected below the threshold, the standing waves and the above special solutions essentially cover all possible orbits:

\begin{teo}[Classification of threshold solutions]\label{sub_class_thresh}
Let $(u,v)$ be a solution to \eqref{sys_NLS} such that $\mathcal{ME}(u_0,v_0) = 1$. Then, the following holds.
\begin{itemize}
    \item[(i)] If $\mathcal{MK}(u_0,v_0) < 1$, then $(u,v)$ is global. Moreover, either $u$ scatters in both time directions, or there exists a ground state $Q=(\phi,\psi)\in \mathcal{G}$ such that $(u,v) = Q^-$  up to the symmetries of the system.

    \item[(ii)] If $\mathcal{MK}(u_0,v_0) = 1$, then there exists a ground state $Q=(\phi,\psi)\in \mathcal{G}$ such that $(u,v) = e^{it}Q$ up to the symmetries of the system.
    
    \item[(iii)] If $\mathcal{MK}(u_0,v_0) > 1$ and $(u_0,v_0)$ is either radial or $(|x|u_0,|x|v_0) \in L^2(\Real^3)\times L^2(\Real^3)$, then either $(u,v)$ blows-up in finite positive and negative time, or there exists a ground state $Q=(\phi,\psi)\in \mathcal{G}$ such that $u = Q^+$ up to the symmetries of the system.
\end{itemize}
\end{teo}

In the above theorem, and throughout the whole paper, by \textit{up to the symmetries} we mean up to scaling, space translation, time translation, phase rotation and time reversal. That is, $(\tilde u,\tilde v) = (u,v)$ up to symmetries if  
$$
 (\tilde u, \tilde v) = (e^{i\theta_0}u_\delta\left(x+x_0,t+t_0 \right), e^{i\theta_1}v_\delta\left(x+x_0,t+t_0 \right)) 
 $$
 or
 $$
 (\tilde u, \tilde v) =(e^{i\theta_0}\bar{u}_\delta\left(x+x_0,\-t+t_0 \right),e^{i\theta_1}\bar{v}_\delta\left(x+x_0,\-t+t_0 \right)),
$$
with $(\delta,\theta_0, \theta_1, x_0,t_0) \in \Real_+\times \Real/2\pi\mathbb{Z}\times \Real/2\pi\mathbb{Z}\times  \Real^N \times \Real$. All these symmetries leave the $\dot{H}^{\frac{1}{2}}\times \dot{H}^{\frac{1}{2}}$ norm invariant.

Threshold problems in the scalar case were studied, for instance, in \cites{DM_Dyn, DR_Thre, higher_thre, CFR_thre}. Even though the general idea of threshold behavior for NLS equations can be considered understood, the main challenge to proving the results above is to deal with the possibility of multiple ground states in system \eqref{sys_NLS}. One has to be careful from the very beginning, since the linearized equation depends on the ground state being considered. The variational characterization given by applying the Sobolev-adapted profile decomposition has to take the multiplicity into account, and the modulational stability proved along the manuscript must be strong enough as to prove that not only the corresponding phase and translation parameters converge as $t \to +\infty$, but that the corresponding ground state (which can be different for different time intervals, as there can be no uniqueness) also converge. Non-degeneracy of the ground states then play an important role here, as well as the spectral properties and coercivity of the linear operator in a suitable subspace.

This paper is structured as follows: In Section \ref{Sec2}, we state the notations used throughout the whole text and recall some basic estimates. In Section \ref{Sec3}, we study the spectral properties of the linearized operator associated to the ground states. Section \ref{Sec4} is devoted to a detailed study of the variational characterization of the ground states, which culminates in a modulation theory developed in Section \ref{Sec5}. We then construct special solutions in Section \ref{Sec6} and prove results about the possible dynamics at the threshold in Section \ref{Sec7}, showing that some of them are associated with exponentially time-decaying solutions to the linearized equation. Finally, Section \ref{Sec8} is devoted to studying such decaying solutions and collecting all of the previous results to close the main results of this paper.

\smallskip

{\bf Acknowledgements.} L. C. was financed by grant \#2020/10185-1, São Paulo Research Foundation (FAPESP). A. P. is partially supported by CNPq grant 303762/2019-5 and FAPESP
grant 2019/02512-5.

\section{Notation and preliminaries}\label{Sec2}

In this section we introduce some notations and give some basic estimates.  We use $c$ to denote various constants that may vary line by line. Given two positive numbers $a$ and $b$, the notation $a \lesssim b$ means
that $a \leq cb$, where the constant $c$ is independent of $a$ and $b$. Sometimes we use $a\lesssim_{a_1,\ldots,a_k}b$ to indicate that the implicit constant $c$ depends on the parameters $a_1,\ldots,a_k$. Given a complex number $z$, we use $\Re(z)$ and $\Im(z)$ to denote, respectively, the real and imaginary parts of $z$. To simplify notation,  $\int f$ always mean integration  over all $\Real^3$. By $\partial_{x_k}$ (or $\partial_k$, for short) we mean differentiation with respect to $x_k$.

By $L^p(\Real^3)$, $1\leq p\leq \infty$, we denote the usual Lebesgue space endowed with the standard norm denoted by $\|\cdot\|_{L^p_x}$ or $\|\cdot\|_{L^p}$. Given a function $f=f(x,t)$, the mixed norm $L^q_tL^r_x$  is defined as $\|f\|_{L^q_tL^r_x}=\|\|f\|_{L^r_x}\|_{L^q_t(\Real)}$. By $\|\cdot\|_{L^q_IL^r_x}$ we mean integration in time over the interval $I\subset\Real$ instead of $\Real$. For any $s\in\Real$, the operators $D^s$ and $J^s$ stand, respectively, for the Fourier
multiplier with symbol $|\xi|^s$ and $\langle \xi \rangle^s = (1+|\xi|)^s$.
Consequently, the norm in the $L^2$-based Sobolev spaces $H^s:=H^s(\Real^3)$ and $\dot{H}^s:=\dot{H}^s(\Real^3)$
are given, respectively, by
\begin{equation*}
\|f\|_{H^s}\equiv \|J^sf\|_{L^2_x}=\|\langle \xi
\rangle^s\hat{f}\|_{L^2_{\xi}}, \qquad \|f\|_{\dot{H}^s}\equiv
\|D^sf\|_{L^2_x}=\||\xi|^s\hat{f}\|_{L^2_{\xi}}.
\end{equation*}
The product space $H^s\times H^s$ is equipped with the norm $\|(f,g)\|^2_{H^s\times H^s}\equiv \|f\|^2_{H^s}+\|g\|^2_{H^s}$. In an analogous way we define the norm in $\dot{H}^s\times \dot{H}^s$.  To simplify notation, sometimes we use $(H^s(\Real^3))^k$ for $H^s(\Real^3)\times \ldots\times H^s(\Real^3)$ ($k$ times).

We say that a pair $(q,r)$ is
$\dot{H}^s$ admissible if
$$
\frac{2}{q}+\frac{3}{r}=\frac{3}{2}-s.
$$
In what follows, we set
$$
\|u\|_{S(L^2)}=\sup\left\{\|u\|_{L^q_tL^r_x};\; (q,r)\, {\rm is}\, L^2\; {\rm
admissible},\; 2\leq r\leq6,\; 2\leq q\leq\infty\right\}
$$
and
$$
\|u\|_{S'(L^2)}=\inf\left\{\|u\|_{L^{q'}_tL^{r'}_x};\; (q,r)\,{\rm is}\, L^2\;
{\rm admissible}, \, 2\leq q\leq\infty,\; 2\leq r\leq6\right\}.
$$

If a time interval $I\subset\Real$ is given, we use ${S(L^2;I)}$ and ${S'(L^2;I)}$ to inform
that the integration in time is over $I$. The space $S(L^2)\times S(L^2)$ is endowed with the norm $\|(u,v)\|_{S(L^2)\times S(L^2)}\equiv \|u\|_{S(L^2)}+\|v\|_{S(L^2)}$.

Next, we recall the well-known Strichartz inequalities. For the proofs we refer the reader to \cite{Cazenave03}.

\begin{lemma}\label{strilem}
The following estimates hold:
\begin{itemize}
    \item[(i)] (Linear estimates)
    $$
    \|e^{it\Delta}u_0\|_{S(L^2)}\lesssim\|u_0\|_{L^2}.
    $$
    \item[(ii)] (Inhomogeneous estimates)
    $$
   \left\| \int_0^te^{i(t-t')\Delta}f(\cdot,t')dt'\right\|_{S(L^2)}\lesssim\|f\|_{S'(L^2)}.
   $$
\end{itemize}
\end{lemma}

As in the scalar case, the ground states of \eqref{sys_def_ground} may also be obtained as minimizers of a  vectorial Gagliardo-Nirenberg inequality. In fact, let us introduce the functional
\begin{equation}\label{Pdef}
    P(u,v)=\int\left[ |u|^4+2\beta|uv|^2+|v|^4\right].
\end{equation}

The Gagliardo-Nirenberg inequality then reads as follows:

\begin{lemma} \label{gnlemma}
For any $(u,v)\in H^1(\Real^3)\times H^1(\Real^3)$ there holds
\begin{equation}\label{GN}
\begin{split}
P(u,v)\leq c_{GN} M(u,v)^{1/2}K(u,v)^{3/2}.
\end{split}
\end{equation}
In addition, the sharp constant $c_{GN}$ is given by
\begin{equation}\label{sharpc}
c_{GN}=\frac{4}{3M(\phi,\psi)^{1/2}K(\phi,\psi)^{1/2}},
\end{equation}
where $(\phi,\psi)$ is any ground state in $\mathcal{G}$.
\end{lemma}

For the proof, see \cite{FanelliMontefusco}. Note that, using integration by parts and \eqref{sys_def_ground}, we may obtain
the following {\textit{Pohozaev}}-type identities
 \begin{align}\label{pohozaev}
      K(\phi,\psi)&= 3 M(\phi,\psi) \quad \mbox{and}\quad
     P(\phi,\psi) = 4M(\phi,\psi).
 \end{align}
 In particular, by setting $(u,v)=(\phi,\psi)$ in \eqref{GN} we see that equality holds. Also, using \eqref{pohozaev} it may be easily checked that
\begin{equation}\label{sharpc2}
c_{GN}=\frac{4}{3\sqrt6 M(\phi,\psi)^{1/2}E(\phi,\psi)^{1/2}},
\end{equation}
which gives $c_{GN}$ in terms of the mass and the energy of the ground states.

\section{The linearized equation}\label{Sec3}

In order to prove the main theorems of this paper, we need to carefully study \eqref{sys_NLS} around the ground state.
We identify the pair of complex numbers $(a + bi,c+di)$ with the vector $
(a,c,b,d)^T
$. For a complex-valued function, we write $f_1$ for its real part and $f_2$ for its complex part. We next introduce the following definition. 

\begin{de}\label{defilinearop} For any ground state $Q = (\phi,\psi)$, we define
\begin{align}\label{def_op}
L_R &=\begin{pmatrix}1-\Delta-3\phi^2 -\beta \psi^2& -2\beta\phi\psi 
\\
-2\beta\phi\psi 
& 1-\Delta-3\psi^2-\beta \phi^2 \\
\end{pmatrix},\\
L_I&=\begin{pmatrix}1-\Delta-\phi^2-\beta \psi^2 & 0 \\
0 & 1-\Delta-\psi^2-\beta \phi^2 \\
\end{pmatrix},\\
\mathcal{L} &:= \begin{pmatrix}0 & -L_I \\
L_R & 0 \\
\end{pmatrix},\\
L(h,k) &= \begin{pmatrix}2\phi^2 h + \phi^2 \overline{h}+ \beta \psi^2 h
+\beta \phi \psi k + \beta \phi \psi \overline{k} 
\\
2\psi^2 k + \psi^2 \overline{k}+ \beta \phi^2 k
+\beta \phi \psi h + \beta \phi \psi \overline{h}
\end{pmatrix} \\
R(h,k) &= \begin{pmatrix}2\phi|h|^2+\phi h^2+\beta\phi|k|^2+\beta\psi h k +\beta \psi h \overline{k}+|h|^2h+\beta|k|^2h\\
2\psi |k|^2+\psi k^2+\beta\phi|h|^2+\beta\phi h k +\beta \psi \overline{h} k+|k|^2k+\beta|h|^2k \end{pmatrix}.
\end{align} 
\end{de}

If $(u,v)$ is a solution to \eqref{sys_NLS}, write $(u,v) = (e^{it}(\phi+h),e^{it}(\psi+k))$. Then $w:= (h,k)$ must satisfy
\begin{equation}\label{linearized_eq}
	\partial_t w + \mathcal{L}w  = iR(w),
\end{equation}
or, writing it as a Schrödinger equation,
\begin{equation}\label{linearized_eq_2}
	i\partial_t w + \Delta w -w +L(w) + R(w) = 0. 
\end{equation}

In the next two sections we establish some properties of the operator $\mathcal{L}$.
\subsection{Properties of the linearized operator}
We have, by a direct calculation, for any ground state $Q = (\phi,\psi)$
\begin{align}
    L_I(\phi,\psi) &= 0,\\
    L_R(\partial_k \phi, \partial_k \psi)&= 0,\quad  1 \leq k \leq 3.
\end{align}
This implies
\begin{equation}
 \mathcal{L}(\partial_k\phi,\partial_k \psi)=\mathcal{L}(i\phi,i\psi) = 0,\quad  1 \leq k \leq 3.
\end{equation}

Also, defining $\Lambda f$ as the scaling generator $f +x \cdot \nabla f$, we have
\begin{equation}
    L_R(\Lambda \phi,\Lambda\psi) = -2(\phi,\psi).
\end{equation}



Moreover, we define the bilinear form associated to $i\mathcal{L} = \begin{pmatrix}L_R&0\\0&L_I\end{pmatrix}$,
\begin{align}
	B((f,g),(\tilde{f},\tilde{g})) &:= \frac{1}{2}(L_R(f_1,g_1),(\tilde{f}_1,\tilde{g}_1))_{L^2}+ \frac{1}{2}(L_I(f_2,g_2),(\tilde{f}_2,\tilde{g}_2))_{L^2},
\end{align}
and the \textit{linearized energy},
\begin{align}
\Phi(f,g) &:= B((f,g),(f,g)) 
\end{align}

One can check directly that, for any $(f,g)$, $(\tilde f, \tilde g) \in S(\Real^3)\times S(\Real^3)$,
\begin{equation}\label{prop_B}
	\begin{gathered}
	B((f,g),(\tilde f, \tilde g)) = B((\tilde f,\tilde g),(f,g)),\\
	B(\mathcal{L}(f,g),(\tilde f, \tilde g)) = - B((f,g),\mathcal{L}(\tilde f, \tilde g)),\\
	\quad B((i\phi,i\psi),(f,g)) = 0,\\
	B((\partial_k \phi,\partial_k \psi),(f,g)) = 0, \,1\leq k \leq 3,\\
	 {B((\Lambda \phi,\Lambda \psi),(f,g)) = -\int (\phi f_1}+\psi g_1).\\
	\end{gathered}.
\end{equation}

We prove the following result:
\begin{lemma}[]\label{spectrum_L} For any ground state of \eqref{sys_NLS}, let $\sigma(\mathcal{L})$ be the spectrum of the operator $\mathcal{L}$, defined in $\left(L^2(\Real^3)\right)^4$ with domain $\left(H^2(\Real^3)\right)^4$ and let $\sigma_{ess}(\mathcal{L})$ be its essential spectrum. Then
\begin{equation}
	\sigma_{ess}(\mathcal{L}) = \{i y:\,\, y \in \Real, |y| \geq 1\},\quad \sigma \cap \Real = \{-e_0,0,e_0\} \quad \text{with } e_0 >0. 
\end{equation}
Moreover, $e_0$ and $-e_0$ are simple eigenvalues  of $\mathcal{L}$ with eigenfunctions $
\mathcal{Y}_+
$, $
\mathcal{Y}_-
\in \mathcal{S}(\mathbb{R}^3) \times \mathcal{S}(\mathbb{R}^3)$, respectively. The null space of $\mathcal{L}$ is spanned by $(i\phi,i\psi)$ and $(\partial_k \phi,\partial_k \psi)$, $1 \leq k \leq 3$ in the case $\beta<1$ and by $(i\phi,0)$, $(0,i\psi)$ and $(\partial_k \phi,\partial_k \psi)$, $1 \leq k \leq 3$ in the case $\beta >1$.

\begin{re}\label{re_spectrum}
By Lemma \ref{spectrum_L}, if $
\mathcal{Y}_1
= 
\Re(\mathcal{Y}_+)
$ and $
\mathcal{Y}_2
= 
\Im(\mathcal{Y}_+)
$, then
\begin{equation}
	L_R 
	\mathcal{Y}_1
	= e_0 
	\mathcal{Y}_2
	\quad \text{and } L_I
	\mathcal{Y}_2
	= -e_0 
	\mathcal{Y}_1
	.
\end{equation}
Furthermore, the null space of $L_R$ is spanned by the vectors $(\partial_k \phi,\partial_k \psi)$, $1 \leq k \leq 3$, and the null space of $L_I$ is spanned by $(\phi,\psi)$ in the case $\beta<1$ and by $(\phi,0)$ and $(0,\psi)$ in the case $\beta > 1$. Moreover, we have
\begin{equation}
    \Phi(\mathcal{Y}_+) = \Phi(\mathcal{Y}_-) = 0.\\
\end{equation}
\end{re}
\end{lemma}

\begin{proof}
{
{The proof follows very closely the scalar case (c.f. \cites{DM_Dyn, DR_Thre, CFR_thre}), and we just sketch the main points here.}
We first show that $\mathcal{L}$ is a relatively compact perturbation of $i(1-\Delta)$. Indeed, define $K: (L^2)^4 \to (L^2)^4 $ as to satisfy $(1-\Delta)^{-1/2}i\mathcal{L}(1-\Delta)^{-1/2} = I-K$. We claim that $K$ is compact. It is enough to show that $\chi (1-\Delta)^{-1/2}$ is a compact operator in $(L^2)^4$ if $\chi$ is a  radial, positive, non-increasing Schwartz function, which follows from the fact that, for any $f$,
\begin{equation}
    \|\chi (1-\Delta)^{-1/2}f\|_{L^2} + \|\nabla(\chi (1-\Delta)^{-1/2}f)\|_{L^2}+\frac{1}{\chi(R)}\|\chi (1-\Delta)^{-1/2}f\|_{L^2(\{|x|\geq R\}}\lesssim \|f\|_{L^2}.
\end{equation}
The assertion about the essential spectrum then follows from Weyl's criterion. We now prove that $\mathcal{L}$ has only one negative eigenvalue (hence, only one positive eigenvalue, by conjugation).

We define the Weinstein functional associated to the corresponding Gagliardo-Nirenberg inequality:
\begin{equation}
J(u,v) = \frac{ M(u,v)^{\frac{1}{2}}K(u,v)^{\frac{3}{2}}}{P(u,v)}.    
\end{equation}

Since $(\phi,\psi)$ is a minimizer of $J$ (see Lemma \ref{gnlemma}), we have $\frac{d^2}{d\eta^2}J(\phi+\eta h,\psi+\eta k)_{|\eta = 0} \geq 0$ for all $(h,k) \in H^1 \times H^1$. A direct calculation then shows that if
\begin{equation}
     \int h_1 \Delta \phi +  k_1 \Delta \psi = 0,
\end{equation}
then
\begin{equation}
    \Phi(h,k) \geq 2 \left(\int \phi h_1 + \psi k_1\right)^2.
\end{equation}

We conclude that $L_R$ can have at most one negative direction, and that $L_I$ has none.

Defining
\begin{equation}
    Z = \Lambda Q - \frac{(Q,\Lambda Q)_{L^2\times L^2}}{(Q,Q)_{L^2\times L^2}}Q
\end{equation}
and noting that $(L_RZ,Z)_{L^2\times L^2}<0$ and $(Z,Q)_{L^2\times L^2}=0$, we conclude that $\mathcal{L}$ has exactly one negative direction. Finally, the assertions about the kernel for $\beta>1$ follow from the non-degeneracy of $L_R$, shown in \cite{WY_Uniq}*{Corollary 4.4}, together with the explicit form of $L_I$, namely
\begin{equation}
    L_I = \begin{pmatrix}1-\Delta-\varphi^2 & 0 \\ 0 & 1-\Delta-\varphi^2 \end{pmatrix},
\end{equation}

since it is known (\cite{W85_stability}*{Proposition 2.8}) that $\ker(1-\Delta-\varphi^2) = \text{span}\{\varphi\}$.

For $0<\beta<1$, we recall that the ground states are the semi-trivial solutions $(\varphi,0)$ and $(0,\varphi)$, where $\varphi$ solves \eqref{scalar_ground}. For $Q = (\varphi,0)$, the operators $L_R$ and $L_I$ reduce to:
\begin{align}
    L_R &=\begin{pmatrix}1-\Delta-3\varphi^2 & 0 \\
0 & 1-\Delta-\beta \varphi^2 \\
\end{pmatrix},\\
L_I&=\begin{pmatrix}1-\Delta-\varphi^2 & 0 \\
0 & 1-\Delta-\beta \varphi^2 \\
\end{pmatrix}.
\end{align}
Since the linearized system is now decoupled, we just need to analyze the operator $L_\gamma := 1-\Delta - \gamma \varphi^2$, $\gamma >0$. It is known (\cite{W85_stability}*{Proposition 2.8}) that $L_1\geq 0$ on $H^1(\mathbb R^3)$, from which we conclude, for $0<\beta<1$:
\begin{equation}
    L_\beta = \beta L_1 + (1-\beta)(1-\Delta) \geq (1-\beta)(1-\Delta) >0.
\end{equation}

Therefore, $\ker({L_\beta}) =\emptyset$. The same Proposition 2.8 of \cite{W85_stability} gives $\ker(L_1) = \text{span}\{\varphi\}$ and $\ker(L_3) = \text{span}\{\nabla \varphi\}$. The proof for $Q = (0,\varphi)$ is analogous.}
\end{proof}


For any ground state $Q = (\phi,\psi)$, consider the following orthogonality relations 
\begin{equation}\label{sub_o1}
\int  h_2\phi = \int k_2\psi  = \int h_1 \partial_j \phi + k_1 \partial_j \psi= 0, \quad 1 \leq j \leq 3,    
\end{equation}
\begin{equation}\label{sub_o2}
    \int h_1 \Delta \phi +  k_1 \Delta \psi = 0.
\end{equation}
\begin{equation}\label{sub_o3}
    B((h,k),\mathcal{Y}_+) = B((h,k),\mathcal{Y_-})=0.
\end{equation}
Denote by $G^\perp$ the set of all $(h,k) \in H^1\times H^1$ satisfying \eqref{sub_o1} and \eqref{sub_o2}, and $\tilde{G}^\perp$ the set of all $(h,k) \in H^1 \times H^1$ satisfying \eqref{sub_o1} and \eqref{sub_o3} (note that either the first or the second member of \eqref{sub_o1} is trivially zero if $\beta < 1$).

By direct calculations, one sees that, for $\beta > 0$,
\begin{equation}
	\Phi_{|\text{span}\{(\nabla \phi,\nabla\psi), (i\phi,i\psi)\}} = 0, \text{ if } \beta < 1,  \quad
	\Phi_{|\text{span}\{(\nabla \phi,\nabla\psi), (i\phi,0), (0,i\psi)\}} = 0, \text{ if } \beta > 1,
\end{equation}
and
\begin{equation}\label{sub_Phi_W_neg}
{\Phi(\phi,\psi) = -2P(\phi,\psi) < 0.}
\end{equation}

Next lemma shows the coercivity property of $\Phi$ on $G^\perp\cup\tilde{G}^\perp$:

\begin{lemma}\label{lem_coerc}
There exists a constant $c > 0$ such that, for any $(f,g) \in G^\perp\cup\tilde{G}^\perp$
\begin{equation}\label{Phi_coerc}
	\Phi(f,g) \geq c\|(f,g)\|_{H^1 \times H^1}^2.
\end{equation}
\end{lemma}

\begin{proof}
{{We also give just a sketch of the proof here.} From the proof of Lemma \ref{spectrum_L}, the relative compactness of the operator $K$ implies that the point spectrum of $i\mathcal{L}$ can only accumulate at $1$. Therefore, only a finite number of eigenvalues (counted with multiplicity) lie in the interval, say, $(-\infty,\tfrac{1}{2}]$. The proof also shows that the first eigenvalue of $i\mathcal{L}$ is negative, and that the second one is zero and has multiplicity $4$ in the case $\beta < 1$, corresponding to $(i\phi,i\psi)$ and $(\partial_k\phi, \partial_k \psi)$, $1\leq k \leq 3$, and multiplicity $5$ in the case $\beta > 1$, corresponding to $(i\phi,0)$, $(0,i\psi)$ and $(\partial_k\phi, \partial_k \psi)$, $1\leq k \leq 3$. The third one then must be positive, implying that $\Phi$ is coercive on $G^{\perp}$. The coercivity on $\tilde{G}^{\perp}$ follows from the index being independent of the basis, from $B(\mathcal{Y}_+,\mathcal{Y}_-) \neq 0$ and from a dimensional counting argument.  }
\end{proof}

\section{Variational characterization of the ground state}\label{Sec4}

The main goal of this section is to give a variational characterization of the ground states. The results here show which functions can lie at the mass-energy threshold and have kinetic energy close to the one of the ground states. This is an essential step towards the modulation theory developed in the following section. We remark that, differently from the scalar case, the ground state here may fail to be unique, and the statements of the propositions below must account for this multiplicity.

\subsection{Bubble decomposition adapted to Gagliardo-Nirenberg-Sobolev for systems of NLS equations}
    
 We first make use of the following \textit{bubble decomposition}, adapted to systems of equations, as to use the same translation parameter in both coordinates. The proof is analogous to the scalar case, to which we refer the reader to \cites{HK, KV_Clay}. A similar decomposition, although adapted to the linear Schr\"odinger evolution, can be found in \cite{FP}.

\begin{prop}\label{bubble_decomp} Let $(f_n,g_n)$ be a bounded sequence in $H^1 \times H^1$. There there exist $J^* \in \{0,1,2,\cdots \}$ $\cup \{ \infty\}$, $\{\phi^j,\psi^j\}_{j=1}^{J^*} \subset H^1 \times H^1$ and $\{x^j_n\}_{j=1}^{J^*}\subset \mathbb R^N$ so that, along a subsequence in $n$, one can write
\begin{align}
    f_n = \sum_{j=1}^{J}\phi^j(x-x_n^j) + r_n^J\\
    g_n = \sum_{j=1}^{J}\psi^j(x-x_n^j) + \tilde{r}_n^J,
\end{align}
for all $1 \leq J \leq J^*$, where
\allowdisplaybreaks
\begin{align}
    \limsup_{J \to J^*} \limsup_n \left(\|r^J_n\|_{L^4} +\|\tilde{r}^J_n\|_{L^4}\right) = 0,\\
    \sup_{J } \limsup_n \left|\|f_n\|^2_{H^1}-\left(\sum_{j=1}^J \|\phi^j\|_{H^1}^2 + \|r_n^J\|_{H^1}^2\right)\right| = 0,\\
    \sup_{J } \limsup_n\left|\|g_n\|^2_{H^1}-\left(\sum_{j=1}^J \|\psi^j\|_{H^1}^2 + \|\tilde{r}_n^J\|_{H^1}^2\right)\right| = 0,\\
     \limsup_{J \to J^*} \limsup_n\left|\|f_n\|^4_{L^4}-\left(\sum_{j=1}^J \|\phi^j\|_{L^4}^4 \right)\right| = 0,\\
     \limsup_{J \to J^*} \limsup_n\left|\|g_n\|^4_{L^4}-\left(\sum_{j=1}^J \|\psi^j\|_{L^4}^4 \right)\right| = 0.
\end{align}
Moreover, if $j \neq k$, then $|x_n^j-x_n^k| \to +\infty$. If $J^*$ is finite, our notation reads $\displaystyle\limsup_{J \to J^*}a_J := a_{J^*}$.
\end{prop}

\begin{re} Note that we have the same translation parameters $x_n^j$ for both $\phi^j$ and $\psi^j$. That implies
\begin{equation}
 \limsup_{J \to J^*} \limsup_n\left|P(f_n,g_n) - \sum_{j=1}^J P(\phi^j,\psi^j)  \right| = 0.
\end{equation}
\end{re}


We now make use of the decomposition above to establish a variational characterization of the ground states.
\begin{prop}\label{prop_mod_1} Let $(f_n,g_n)$ be a sequence in $H^1 \times H^1$ such that $M(f_n,g_n) = M(\phi,\psi)$ and $E(f_n,g_n) = E(\phi,\psi)$, where $(\phi,\psi)$ is any ground state to \eqref{sys_NLS}. If 
\begin{equation}
    K(f_n,g_n) \to K(\phi,\psi),
\end{equation}
then, up to a subsequence, there exist $\theta_0, \theta_1 \in \mathbb{R}/2\pi \mathbb{Z}$ and $\{x_n\} \subset \mathbb R^N$ such that
\begin{equation}\label{mod_close_H1}
    \min_{(\phi,\psi)\in \mathcal{G}}\left\|\left(e^{i\theta_0}f_n(\cdot-x_n),e^{i\theta_1}g_n(\cdot-x_n)\right) - (\phi,\psi)\right\|_{H^1 \times H^1} \to 0.
\end{equation}
\end{prop}
\begin{proof} If \eqref{mod_close_H1} does not hold for any choice of $\theta_0$ and $\{x_n\}$, passing to a subsequence, if necessary, we apply the bubble decomposition to write, as in Propostion \ref{bubble_decomp},
\begin{align}
    f_n = \sum_{j=1}^{J}\phi^j(x-x_n^j) + r_n^J\\
    g_n = \sum_{j=1}^{J}\psi^j(x-x_n^j) + \tilde{r}_n^J.
\end{align}
Using the fact that $E(f_n,g_n) = E(\phi,\psi)$, we deduce
\begin{equation}\label{bubble_modulation}
    P(f_n,g_n) = 2K(f_n,g_n)-4E(f_n,g_n)\to 2K(\phi,\psi)-4E(\phi,\psi) = P(\phi,\psi).
\end{equation}
By the profile decoupling, this means
\begin{equation}
    \sum_{j=1}^{J^*}P(\phi^j,\psi^j) = P(\phi,\psi).
\end{equation}
On the other hand, by the $H^1$ decoupling, for all $J$,
\begin{align}
    M(\phi^k,\psi^k) \leq  \sum_{j=1}^J M(\phi^j,\psi^j) &\leq M(\phi,\psi), \text{ for } k \leq J,
\end{align}
and
\begin{align}
    \sum_{j=1}^J K(\phi^j,\psi^j) &\leq K(\phi,\psi),
\end{align}
which implies
\begin{align}
    \sup_{J} M(\phi^k,\psi^k) \leq M(\phi,\psi) \quad \mbox{and} \quad \sum_{j=1}^{J^*} K(\phi^j,\psi^j) \leq K(\phi,\psi).
\end{align}
Therefore, by the sharp Gagliardo-Nirenberg inequality,
\begin{equation}\label{ineq_1}
    P(\phi,\psi) = \sum_{j=1}^{J^*}P(\phi^j,\psi^j) \leq c_{GN}\sum_{j=1}^{J^*} M(\phi^j,\psi^j)^{\frac{1}{2}}K(\phi^j, \psi^j)^\frac{3}{2} \leq c_{GN}M(\phi,\psi)^{\frac{1}{2}}\sum_{j=1}^{J^*} K(\phi^j, \psi^j)^\frac{3}{2}.
\end{equation}
Since 
\begin{equation}\label{ineq_2}
    \sum_{j=1}^{J^*} K(\phi^j, \psi^j)^\frac{3}{2} \leq \left[\sum_{j=1}^{J^*} K(\phi^j, \psi^j)\right]^\frac{3}{2} \leq K(\phi,\psi)^\frac{3}{2},
\end{equation}
the uniqueness up to symmetries of the set of minimizers of the sharp Gagliardo-Nirenberg inequality ensures that all the inequalities in \eqref{ineq_1} and \eqref{ineq_2} are equalities, and also implies that $J^*\leq1$, otherwise the equality in \eqref{ineq_2} cannot hold. The case $J^*=0$ is precluded since $P(\phi,\psi)\neq 0$.

Therefore, up to constant phase and translation, $\phi^1 = \phi$, $\psi^1$ = $\psi$ for some ground state $(\phi,\psi)$. The decoupling in \eqref{bubble_modulation} then ensures $\|(r_n^1,\tilde{r}_n^1)\|_{H^1 \times H^1} \to 0$. We have thus found $\tilde{\theta}_0, \tilde{\theta}_1$ and $\{\tilde{x}_n\}$ such that 
\begin{equation}
     \min_{(\phi,\psi)\in \mathcal{G}}\left\|\left(e^{i\tilde{\theta}_0}f_n(\cdot-\tilde{x}_n),e^{i\tilde{\theta}_1}g_n(\cdot-\tilde{x}_n)\right) - (\phi,\psi)\right\|_{H^1 \times H^1}  \to 0,
\end{equation}
contradicting our first assumption that \eqref{mod_close_H1} does not hold.
\end{proof}

We now define the quantity
\begin{equation}\label{def_delta}
    \delta(f,g) = |K(f,g) -K(\phi,\psi)|. 
\end{equation}
and note that Proposition \ref{prop_mod_1} gives:

\begin{cor}\label{modulation_quant_1}
If $M(f,g) = M(\phi,\psi)$ and $E(f,g) = E(\phi,\psi)$, then
\begin{equation}
    \inf_{\theta_0, \theta_1,y}  \min_{(\phi,\psi)\in \mathcal{G}} \left\|\left(e^{i\theta_0}f(\cdot-y),e^{i\theta_1}g(\cdot-y)\right) - (\phi,\psi)\right\|_{H^1 \times H^1} <\varepsilon(\delta(f,g)),
\end{equation}
 where $\varepsilon(d) \to 0$ as $d \to 0$.
\end{cor}




\section{Modulation theory}\label{Sec5}

We are now able to construct and refine the modulation theory for the system \eqref{sys_NLS}. Our first result shows that, on connected time intervals, the ground state given by the variational characterization must remain the same.

To abbreviate the notation, if $I \subset \mathbb R$ and $(u,v):I \to H^1\times H^1$, we often write $\delta(t) = \delta(u(t),v(t))$ for $t \in I$. We then have the following time-dependent version of modulation.
\begin{cor}\label{cor:connected_modulation} There exists $\delta_0>0$ such that if $(u,v): I \to H^1 \times H^1$ is a continuous curve on an interval $I = (a,b)$, $M(u(t),v(t)) = M(\phi,\psi)$, $E(u(t),v(t)) = E(\phi,\psi)$ and $\delta(t) <\delta_0$ for all $t \in I$, then there exist $(\phi,\psi)\in\mathcal{G}$, and functions $\theta_0,\theta_1: I \to \mathbb{R}/2\pi\mathbb Z$ and $x_0: I \to \mathbb{R}^3$ such that, for all $t \in I$,
\begin{equation}\label{modulation_param_1}
    \left\|\left(e^{i\theta_0(t)}u(\cdot-x_0(t)),e^{i\theta_1(t)}v(\cdot-x_0(t))\right) - (\phi,\psi)\right\|_{H^1 \times H^1} < \varepsilon(\delta(t)),
\end{equation}
where $\varepsilon(d) \to 0$ as $d \to 0$.
\end{cor}
\begin{proof}
We first note that the manifolds generated by the group of symmetries acting on the ground states are at a positive distance from each other, since all the ground states are non-negative and radial. That is,
\begin{equation}
     \inf_{\substack{x \in \Real^3\\ \theta_0, \theta_1 \in \mathbb{R}/2\pi\mathbb Z}}\min_{\substack{(\phi,\psi)\in \mathcal{G}\\ (\tilde\phi,\tilde\psi)\in \mathcal{G}\\{(\phi,\psi)\neq (\tilde\phi,\tilde\psi)}}}\left\|\left(e^{i\theta_0}\phi(\cdot-x),e^{i\theta_1}\psi(\cdot-x)\right) - (\tilde \phi,\tilde \psi)\right\|_{H^1 \times H^1} > 0.
\end{equation}

This implies that the choice of ground state in \eqref{modulation_param_1} is unique, if $\delta_0>0$ is chosen sufficiently small. We then apply Corollary \ref{modulation_quant_1} to obtain $\theta_0(t)$, $\theta_1(t)$ and $x_0(t)$ for every $t \in I$. 
\end{proof}

\begin{re} In the case $0< \beta < 1$, if $(\phi,\psi) = (\varphi,0)$, one can choose $\theta_1 \equiv 0$, since $v$ itself is small in the $H^1$ norm. The similar choice $\theta_0 \equiv 0$ can be made in the case $(\phi,\psi) = (0,\varphi)$.
\end{re}

We now exploit the coercivity of the linearized operator to modify the modulation parameters in order to get control on their size depending on $\delta(u(t),v(t))$.
\subsection{Modulation with orthogonality}
\begin{prop}\label{prop_mod_ort}
Let $\delta_0 >0$ and $I_0$ be the set
\begin{equation}
    I_0 = \{t : \delta(t)< \delta_0\}.
\end{equation} 

Then, for every connected interval $J \subset I_0$, there exist  $(\phi,\psi)\in \mathcal{G}$, $\theta, \tilde\theta: J \to \mathbb{R}/2\pi\mathbb Z$, $x: J \to \mathbb{R}^3$, $\alpha: J \to \mathbb{R}$ and $(h,k) : J \to H^1 \times H^1 $ such that, for all $t \in J$,
\begin{equation}
    \left\|(u(t),v(t))-\left[(1+\alpha(t))\left\{e^{i\theta(t)}\phi(\cdot-x(t)),e^{i\tilde\theta(t)}\psi(\cdot-x(t))\right\}+(h(t),k(t))\right]\right\|_{H^1 \times H^1} < \epsilon(\delta(t)).
\end{equation}

Moreover, in the case $0<\beta<1$, one can choose $\tilde\theta \equiv 0$ if $(\phi, \psi) = (\varphi,0)$ and $\theta \equiv 0$ if $(\phi,\psi)=(0,\varphi)$.
\end{prop}


\begin{proof}
To prove Proposition \ref{prop_mod_ort}, we first make use of the Implicit Function Theorem to modify $\theta_0$ and $x_0$ in order to impose suitable orthogonality conditions. Then, we can use the coercivity of the quadratic form $\Phi$ to obtain bounds on the derivatives of the parameters.

Given a connected component $J \subset I_0$, consider the corresponding $\theta_0$, $\theta_1$, $x_0$ and $(\phi,\psi)$ given by Corollary \ref{cor:connected_modulation}. In the case $0<\beta < 1$, if $(\phi,\psi) = (\varphi, 0)$, define
\begin{align}
    \Omega &: H^1 \times \mathbb R/2\pi \mathbb Z \times \mathbb R^3 \to \mathbb R^4,\\
    \Omega(u,\theta,x) &= \begin{bmatrix}
    \Im (u,e^{i\theta}\varphi(\cdot-x))_{L^2}\\
    \Re (u,e^{i\theta}\partial_{x_1}\varphi(\cdot-x))_{L^2 }\\
    \Re (u,e^{i\theta}\partial_{x_2}\varphi(\cdot-x))_{L^2 }\\
    \Re (u,e^{i\theta}\partial_{x_3}\varphi(\cdot-x))_{L^2 }
    \end{bmatrix}.
\end{align}

For all $t \in J$, note that
\begin{equation}
    \Omega(e^{i\theta_0(t)}\phi(\cdot-x_0(t)),\theta_0(t),x_0(t)) = 0
\end{equation}
and
\begin{equation}
    \det\frac{\partial \Omega}{\partial{(\theta,x)}}\big|_{(e^{i\theta_0(t)}\varphi(\cdot-x_0(t)),\theta_0(t),x_0(t))} = -\left(\int \phi^2\right)\left(\displaystyle\prod_{k=1}^3\int |\partial_{x_k}\varphi|^2\right)\neq 0.
\end{equation}

Note that the right-hand side of the last equations is independent of $t$. Therefore, there exist balls $U_t \subset H^1 $, $V_t\subset \mathbb R/2\pi \mathbb Z \times \mathbb R^3 $ (with diameter bounded below uniformly on $t \in J$) and $\xi_t:U_t \to V_t$ such that for all $t\in J$,   $e^{i\theta_0(t)}\varphi(\cdot-x_0(t)) \in U_t$, $(\theta_0(t),x_0(t))\in V_t$, and $\Omega(u,\xi_t(u,v)) = 0$ for all $u \in U_t$. By choosing a smaller $\delta_0>0$, if necessary, Corollary \ref{cor:connected_modulation} ensures $u(t)\in U_t$ for all $t \in J$. Hence, defining $(\theta(t),x(t)) = \xi_t(u(t))$, we have, for all $t$,
\begin{equation}
     \begin{bmatrix}
    \Im (u(t),e^{i\theta(t)}\varphi(\cdot-x(t)))_{L^2}\\
    \Re (u(t),e^{i\theta(t)}\partial_{x_1}\varphi(\cdot-x(t)))_{L^2}\\
   \Re (u(t),e^{i\theta(t)}\partial_{x_2}\varphi(\cdot-x(t)))_{L^2}\\
   \Re (u(t),e^{i\theta(t)}\partial_{x_3}\varphi(\cdot-x(t)))_{L^2}\\
    \end{bmatrix} = 0.
\end{equation}

Now write
\begin{equation}\label{mod_decomposition}
    (e^{-i\theta(t)}u(x+x(t),t),v(x,t)) = (1+\alpha(t)) (\varphi,0)+ (h(x,t),k(x,t)),
\end{equation}
where
\begin{equation}
    \alpha(t) := \frac{ \Re\displaystyle\int  e^{-i\theta(t)}u(x+x(t),t) \Delta\varphi \, dx}{K(\varphi,0)}-1
\end{equation}
is chosen as to make $(h,k) \in G^\perp$. In this case, we can take $\tilde\theta \equiv 0$. The case $(\phi,\psi) = (0,\varphi)$ is handled analogously.

In the case $\beta>1$, we recall that the ground state is given by $((1+\beta)^{-1/2}\varphi,(1+\beta)^{-1/2}\varphi)$. Since we do not have a semi-trivial ground state, we take advantage of both phase symmetries and define
\begin{align}
    \Omega &: H^1 \times H^1\times \mathbb R/2\pi \mathbb Z \times \mathbb R/2\pi \mathbb Z \times \mathbb R^3 \to \mathbb R^5,\\
    \Omega((u,v),\theta, \tilde\theta,x) &= \begin{bmatrix}
    \Im (u,e^{i\theta}\varphi(\cdot-x))_{L^2}\\
    \Im (v,e^{i\tilde\theta}\varphi(\cdot-x))_{L^2}\\
    \Re (u,e^{i\theta}\partial_{x_1}\varphi(\cdot-x))_{L^2 }+ \Re (v,e^{i\tilde\theta}\partial_{x_1}\varphi(\cdot-x))_{L^2 }\\
     \Re (u,e^{i\theta}\partial_{x_2}\varphi(\cdot-x))_{L^2 }+ \Re (v,e^{i\tilde\theta}\partial_{x_2}\varphi(\cdot-x))_{L^2 }\\ 
     \Re (u,e^{i\theta}\partial_{x_3}\varphi(\cdot-x))_{L^2 }+ \Re (v,e^{i\tilde\theta}\partial_{x_3}\varphi(\cdot-x))_{L^2 }\\
    \end{bmatrix}.
\end{align}

In the same fashion, since $\det\frac{\partial \Omega}{\partial{(\theta,x)}}\big|_{((e^{i\theta_0(t)}\varphi(\cdot-x_0(t)),e^{i\theta_1(t)}\varphi(\cdot-x_0(t))),\theta_0(t), \theta_1(t),x_0(t))} \neq 0$, we obtain $(\theta(t), \tilde\theta(t),x(t))\in \mathbb R/2\pi \mathbb Z \times \mathbb R/2\pi \mathbb Z \times \mathbb R^3 $ such that

\begin{equation}
     \begin{bmatrix}
    \Im (u(t),e^{i\theta(t)}\varphi(\cdot-x(t)))_{L^2}\\
    \Im (v(t),e^{i\tilde\theta(t)}\varphi(\cdot-x(t)))_{L^2}\\
    \Re (u(t),e^{i\theta(t)}\partial_{x_1}\varphi(\cdot-x(t)))_{L^2}+\Re (v(t),e^{i\tilde\theta(t)}\partial_{x_1}\varphi(\cdot-x(t)))_{L^2}\\
   \Re (u(t),e^{i\theta(t)}\partial_{x_2}\varphi(\cdot-x(t)))_{L^2}+\Re (v(t),e^{i\tilde\theta(t)}\partial_{x_2}\varphi(\cdot-x(t)))_{L^2}\\
   \Re (u(t),e^{i\theta(t)}\partial_{x_3}\varphi(\cdot-x(t)))_{L^2}+\Re (v(t),e^{i\tilde\theta(t)}\partial_{x_3}\varphi(\cdot-x(t)))_{L^2}\\
    \end{bmatrix} = 0.
\end{equation}

We then write

\begin{equation}\label{mod_decomposition_2}
    (e^{-i\theta(t)}u(x+x(t),t),e^{-i\tilde\theta(t)}v(x+x(t),t)) = (1+\beta)^{-1/2}(1+\alpha(t)) (\varphi,\varphi)+ (h(x,t),k(x,t)),
\end{equation}

where 
\begin{equation}
    \alpha(t) := \frac{ \Re\displaystyle\int  e^{-i\theta(t)}u(x+x(t),t) \Delta\varphi +  e^{-i\tilde\theta(t)}v(x+x(t),t) \Delta\varphi \, dx}{K(\varphi,\varphi)}-1.
\end{equation}
\end{proof}

We then have the following bounds on the parameters:

\begin{prop} Let $(u,v)$ be a solution to \eqref{sys_NLS} satisfying $M(u,v) = M(\phi,\psi)$ and $E(u,v) = E(\phi,\psi)$. Then, taking a smaller $\delta_0$, if necessary, one has, for $t \in I_0$:
\begin{equation}\label{mod_approx_small}
    |\alpha(t)| \approx \left| \int \phi h_1(t) + \psi k_1(t) \right| \approx \|(h(t),k(t))\|_{H^1 \times H^1} \approx \delta(t).
\end{equation}
One also has
\begin{equation}\label{mod_lesssim_small}
    |\alpha'(t)| + |x'(t)| + \|\phi\|_{L^2}^2|\theta'(t)-1|+\|\psi\|_{L^2}^2|\tilde\theta'(t)-1| \lesssim \delta(t).
\end{equation}
\end{prop}
\begin{proof}
Let $\tilde{\delta} = \delta + |\alpha|+ \|(h,k)\|_{H^1 \times H^1}$. The relation $M(\phi + \alpha\phi + h,\psi + \alpha\psi + k) = M(\phi,\psi)$ gives
\begin{equation}\label{mod_small_1}
    \left||\alpha|-\frac{1}{M(\phi,\psi)}\left|\int \phi h_1 + \psi k_1\right| \, \right| = O(\tilde{\delta}^2).
\end{equation}
Moreover, the orthogonality condition \eqref{sub_o1} gives
\begin{equation}\label{mod_small_2}
    K(\phi + \alpha\phi + h,\psi + \alpha\psi + k)-K(\phi,\psi) =  2\alpha K(\phi,\psi)+\alpha^2 K(\phi,\psi) + K(h,k).
\end{equation}

Thus, by the definition of $\delta$,
\begin{equation}\label{mod_small_3}
    \delta \lesssim |\alpha| +\tilde{\delta}^2 \lesssim  \delta  +\tilde{\delta}^2.
\end{equation}

Note that the orthogonality condition \eqref{sub_o1} gives 
\begin{equation}
    \int \phi h_1 + \psi k_1 = - \int (\phi^2+\beta \psi^2)\phi h_1 +  (\psi^2+\beta \phi^2)\psi k_1.
\end{equation}
Therefore, $B((\phi,\psi),(h_1,k_1)) = - \displaystyle\int \phi h_1 + \psi k_1 $. Now, since the relations  $M(\phi + \alpha\phi + h,\psi + \alpha\psi + k) = M(\phi,\psi)$  and $E(\phi+\alpha \phi + h,\psi+\alpha \phi + k) = E(\phi, \psi)$ give
\begin{equation}
    \Phi(\alpha \phi +h, \alpha \psi +k) \lesssim |\alpha|^3 + \|(h,k)\|_{H^1\times H^1}^3,
\end{equation}
we have
\begin{equation}\label{mod_small_4}
    \alpha^2 \Phi(\phi,\psi) + 2 \alpha B((\phi,\psi),(h_1,k_1))+ \Phi(h,k) = \Phi(\alpha \phi +h, \alpha \psi +k) \lesssim \tilde{\delta}^3,
\end{equation}
or
\begin{equation}
    \Phi(h,k) \lesssim |\alpha|^2 + \tilde{\delta}^3.
\end{equation}
Finally, coercivity on $G^\perp$ implies $\Phi(h,k) \approx \|(h,k)\|_{H^1 \times H^1}^2$, showing that
\begin{equation}\label{mod_small_5}
    \|(h,k)\|_{H^1 \times H^1} \lesssim |\alpha| + \tilde{\delta}^{\frac{3}{2}}.
\end{equation}
Equations \eqref{mod_small_1}, \eqref{mod_small_2} and \eqref{mod_small_5} give $\tilde{\delta}\lesssim |\alpha|$, which in turn shows that \eqref{mod_small_1}-\eqref{mod_small_5} imply \eqref{mod_approx_small}. We not turn to the proof of \eqref{mod_lesssim_small}. 

Let $\delta^* = \delta + |\alpha'|+|x'|+|\theta'-1|+|\tilde\theta'-1|$. We write the equations for $h$ and $k$, given the decomposition \eqref{mod_decomposition}, as 
\begin{equation}
    \begin{cases}
    i \partial_t h + \Delta h +i\alpha' \phi - (\theta'-1) \phi - i \nabla \phi \cdot x'  = O( \delta + \delta \delta^*)\\
    i \partial_t k + \Delta h +i\alpha' \psi - (\tilde\theta'-1) \psi - i \nabla \psi \cdot x' = O( \delta + \delta \delta^*)
    \end{cases}.
\end{equation}
Projecting the pair of equations above against $(\phi,0)$, $(0,\psi)$, $(\partial_j \phi, \partial_j \psi)$, $1 \leq j \leq 3$,  $(\Delta \phi, \Delta \psi)$ and integrating by parts yields 
\begin{equation}
   |\alpha'|+|x'|+\|\phi\|_{L^2}^2|\theta'(t)-1|+\|\psi\|_{L^2}^2|\tilde\theta'(t)-1|\lesssim \delta + \delta \delta^*,
\end{equation}
which implies \eqref{mod_lesssim_small}.
\end{proof}
The following lemma gives a simple criteria to show the convergence of the modulation parameters. It reflects a strong modulational stability around the ground states.
\begin{lemma}\label{const_modul} Let $(u,v)$ be a global solution to \eqref{sys_NLS} such that $M(u,v) = M(\phi,\psi)$ and $E(u,v) = E(\phi,\psi)$. If 
\begin{equation}\label{delta_int_exp}
    \int_0^\infty \delta(s) \, ds < +\infty,
\end{equation}
then there exist $\theta_0$, $x_0$ and $(\phi_0,\psi_0)$ such that 
\begin{equation}\label{const_mod_param}
    \|(e^{-i\theta_0}u(x+x_0,t),e^{-i\theta_1}v(x+x_0,t) - (e^{it}\phi_0,e^{it}\psi_0)\|_{H^1 \times H^1} \lesssim  \int_t^\infty \delta(s) \, ds.
\end{equation}
In other words, up to constant symmetries of the NLS system \eqref{sys_NLS}, the flow $(u(t),v(t))$ approaches the standing wave $(e^{it}\phi_0,e^{it}\psi_0)$ as $t \to +\infty$.
\begin{proof}
We first show that $|\alpha(t)| \approx\delta(t) \to 0$ as $t \to +\infty$. If not, then there exist $\eta>0$ and two unbounded sequences of times, $\{t_n\}$ and $\{t_n'\}$, such that $t_n < t_n'$ for all $n$, $\delta(t) < \delta_0$ for $t \in [t_n,t_n']$, $\delta(t_n) \to 0$ and $\delta(t_n') = \eta$. For all $n$, $\alpha(t)$ is defined for $t \in [t_n,t_n']$ and, by 
\eqref{mod_lesssim_small} and \eqref{delta_int_exp}, it satisfies
\begin{equation}
    |\alpha(t_n')-\alpha(t_n)| \leq \int_{t_n}^{t_n'} |\alpha'(t)| \, dt \lesssim \int_{t_n}^{+\infty} \delta(t) \, dt\to 0,
\end{equation}
which contradicts $\delta(t_n) < \eta/2$ and $\delta(t_n') = \eta $ for all $n$.

Therefore, if $t_0>0$ is large enough, then $\delta(t) < \delta_0$ for $t> t_0$, so that the modulation parameters $\alpha(t)$, $\theta(t)$, $\tilde\theta(t)$ and $x(t)$, as well as the corresponding ground state $(\phi_0,\psi_0)$, are (uniquely) defined for any $t>t_0$ and satisfy
\begin{equation}
    d(t) \approx |\alpha(t)| \leq \int_t^{+\infty} |\alpha'(s)| \, ds \lesssim \int_t^{+\infty} \delta(s) \, ds,
\end{equation}
and, for $\beta >1$,
\begin{align}
    \int_{t_0}^{+\infty}|\theta'(t)-1|+|\tilde\theta'(t)-1| + |x'(t)| \, dt \lesssim \int_{t_0}^{+\infty} \delta(t) \, ds < +\infty.
\end{align}
This means that $\theta(t)-t$, $\tilde\theta(t)-t$ (with the respective changes for $\theta$ or $\tilde\theta$ in the case $0<\beta<1$) and $x(t)$ have limits, say, $\theta_0$, $\theta_1$ and $x_0$ as $t \to +\infty$. By continuity of the flow, we get \eqref{const_mod_param}.
\end{proof}

\end{lemma}

\section{Construction of special solutions}\label{Sec6}

In this section, we prove the existence and uniqueness of special solutions to the NLS system \eqref{sys_NLS} at the mass-energy threshold that exponentially approach the standing waves $(e^{it}\phi,e^{it}\psi)$.
\subsection{Preliminary estimates}

We start with some estimates for the linearized equation.
\begin{lemma}\label{prel_estimates} For any ground state $Q=(\phi,\psi)$ and any time interval $I$ such that $|I|\leq 1$, we have
\begin{itemize}
    \item $\|\langle\nabla\rangle L(w)\|_{S'(L^2,I)\times S'(L^2,I)} \lesssim |I|^{\frac{1}{2}}\|\langle\nabla\rangle w\|_{S(L^2,I)\times S(L^2,I)}$,
    \item $\|\langle\nabla\rangle (R(v+w)-R(v+\tilde{w}))\|_{S'(L^2,I)\times S'(L^2,I)}$  
    
    $\quad\quad\lesssim \left[\|\langle\nabla\rangle v\|_{S(L^2,I)\times S(L^2,I)}
    +
    \|\langle\nabla\rangle w\|_{S(L^2,I)\times S(L^2,I)}
    +
    \|\langle\nabla\rangle \tilde{w}\|_{S(L^2,I)\times S(L^2,I)}]\right.$
    
    $
    \,\quad\quad+\left.\|\langle\nabla\rangle v\|_{S(L^2,I)\times S(L^2,I)}^2
    +
    \|\langle\nabla\rangle w\|_{S(L^2,I)\times S(L^2,I)}^2
    +
    \|\langle\nabla\rangle \tilde{w}\|_{S(L^2,I)\times S(L^2,I)}^2\right]\|\langle\nabla\rangle(w-\tilde{w})\|.
    $
\end{itemize}
where $L$ and $R$ were defined in Definition \ref{defilinearop}.
\end{lemma}
\begin{proof} By the explicit forms of $L$ and $R$, the lemma follows immediately from H\"older and Sobolev's inequalities, since, for any $f,g,h$,
\begin{align}
    \|fgh\|_{L^2_I L^{\frac{6}{5}}_x} &\lesssim |I|^{\frac{1}{2}}\|\nabla f\|_{L^\infty_I L^{2}_x}
    \|\nabla g\|_{L^\infty_I L^{2}_x}
    \|h\|_{L^\infty_I L^{2}_x}, \text{ and}\\
    \|fgh\|_{L^2_I L^{\frac{6}{5}}_x} &\lesssim \||\nabla|^{\frac{1}{2}} f\|_{L^2_I L^{4}_x} 
    \||\nabla|^{\frac{1}{2}} g\|_{L^2_I L^{4}_x}
    \|h\|_{L^\infty_I L^{2}_x}.
\end{align}
\end{proof}

\subsection{Spectral decay and regularity} Let $(F,G) \in \mathcal{S}(\mathbb R^3) \times \mathcal{S}(\mathbb R^3)$ and $\lambda \in \mathbb R$. If, given a ground state $Q=(\phi,\psi)$, $(f,g) \in H^2(\mathbb{R}^3) \times H^2(\mathbb{R}^3)$ satisfies
\begin{equation}\label{spectral_elliptic}
    \mathcal{L}(f,g) - \lambda (f,g) = (F,G),
\end{equation}
then $(f,g) \in \mathcal{S}(\mathbb R^3) \times \mathcal{S}(\mathbb R^3)$.
\begin{proof}
Defining $A$ and $B$ as 
\begin{align}
    L_R = (1-\Delta)I - A,\\
    L_I = (1-\Delta)I - B,
\end{align}
we can rewrite \eqref{spectral_elliptic} more explicitly, obtaining
\begin{align}\label{full_eq_spectral_elliptic}
    \left((1-\Delta)^2+\lambda^2\right)\begin{pmatrix}
     f_1 \\
   g_1
    \end{pmatrix}
    = B L_R\begin{pmatrix}
     f_1 \\
   g_1
    \end{pmatrix} + L_I\left\{A\begin{pmatrix}
     f_1 \\
   g_1
    \end{pmatrix}\right\}-BA\begin{pmatrix}
     f_1 \\
   g_1
    \end{pmatrix}.
\end{align}
with a similar equation for $(f_2,g_2)$. By induction, it is clear that $f,g \in H^s(\mathbb R^3)$, for any $s \geq 0$. As for the decay, let $\chi \in C^\infty_c(\mathbb{R}^3)$ and $\chi_R (x) = \chi(x/R)$, for $R \geq 1$. Let also $\tilde{\chi}\in C^\infty_c(\mathbb{R}^3)$ be such that $\tilde{\chi} \equiv 1$ on the support of $\chi$. We will show that, for any such $\chi$ and any non-negative integers $s$, $l$, one has
\begin{equation}\label{Hs_decay_ind}
    \|\chi_R f\|_{H^s} + \|\chi_R g\|_{H^s} \lesssim_{\chi,s,l} \frac{1}{R^l}.
\end{equation}
It is clear that \eqref{Hs_decay_ind} holds for $(s,l) = (3,0)$. Now, note that
\begin{align}
    \|\chi_R f_1\|_{H^{s+1}} + \|\chi_R g_1\|_{H^{s+1}} &\approx \| \left((1-\Delta)^2+\lambda^2\right)\chi_R f_1\|_{H^{s-3}}
    +\|\left((1-\Delta)^2+\lambda^2\right)\chi_R g_1\|_{H^{s-3}}
    \\&\lesssim
    \|\chi_R \left((1-\Delta)^2+\lambda^2\right) \tilde{\chi}_R f_1\|_{H^{s-3}} 
    +
   \|\chi_R \left((1-\Delta)^2+\lambda^2\right) \tilde{\chi}_R g_1\|_{H^{s-3}} \\
   &\quad+
    \|[ (1-\Delta)^2+\lambda^2;\chi_R] \tilde{\chi}_R f_1\|_{H^{s-3}}
  + \|[ (1-\Delta)^2+\lambda^2;\chi_R] \tilde{\chi}_R g_1\|_{H^{s-3}}.
\end{align}
For the first two terms on the right-hand side, the equation \eqref{full_eq_spectral_elliptic} and the (exponential) decay of the ground state and its derivatives give
\begin{align}
     \|\chi_R \hspace{-2pt}\left((1-\Delta)^2+\lambda^2\right)\hspace{-2pt} \tilde{\chi}_R f_1\|_{H^{s-3}}
     +
      \|\chi_R \hspace{-2pt}\left((1-\Delta)^2+\lambda^2\right) \hspace{-2pt}\tilde{\chi}_R g_1\|_{H^{s-3}}
      &\lesssim \frac{1}{R} \|\tilde{\chi}_R f_1\|_{H^{s-1}} 
      +
      \frac{1}{R} \|\tilde{\chi}_R g_1\|_{H^{s-1}} \\
      &\lesssim \frac{1}{R}\|\tilde{\chi}_R f_1\|_{H^{s}}
      +
      \frac{1}{R}\|\tilde{\chi}_R g_1\|_{H^{s}},
\end{align}
whilst the commutator immediately bounds the remaining terms, since an explicit calculation gives
\begin{equation}
    \|[ (1-\Delta)^2+\lambda^2;\chi_R]\|_{H^s \to H^{s-3}} \lesssim
     \frac{1}{R}.
\end{equation}
Thus, we have just showed that, if \eqref{Hs_decay_ind} holds for $(s,l)$, with $s\geq 3$, then it also holds for $(s+1,l+1)$, which completes the proof for $f_1$ and $g_1$. For $f_2$ and $g_2$, the argument is completely analogous and is ommited.
\end{proof}

\subsection{A family of approximate solutions}

Starting from an eigenfuction of $\mathcal{L}$, we now iteratively construct a family of approximate solutions to the linearized equation.

\begin{prop}\label{family_appr}  Let $A \in \Real$ and $Q=(\phi,\psi)$ be any ground state. There exists a sequence ($Z_k^A)_{l \geq 1}$ of functions in $\mathcal{S}(\Real^3) \times \mathcal{S}(\Real^3)$ such that $Z_1^A = A \mathcal{Y}_+$ and, if $l \geq 1$ and $\mathcal{V}_l^A = \sum_{j=1}^l e^{-je_0 t}Z_j^A$, then as $t \to +\infty$ we have
\begin{equation}\label{eq_sol_v}
\partial_t\mathcal{V}_l^A + \mathcal{LV}_l^A = R(\mathcal{V}_l^A)+O\left(e^{-\left(l+1\right)e_0t}\right) \text{ in }   \mathcal{S}(\Real^3)\times \mathcal{S}(\Real^3).
\end{equation}
\end{prop}

\begin{proof}
The sequence is constructed inductively. We omit the index $A$ throughout this proof. Define $Z_1 = A\mathcal{Y}_+$ and note that
\begin{equation}
    \partial_t \mathcal{V}_1 + \mathcal{LV}_1 -R(\mathcal{V}_1) = -R(\mathcal{V}_1).
\end{equation}
Now, since the entries of $R(f,g)$ are polynomials of degree three on $f, \overline{f},g,\overline{g}$, for any $(f,g)$, and since $\mathcal{Y}_+ \in \mathcal{S}(\Real^3)$, we conclude that 
\begin{equation}
    \partial_t \mathcal{V}_1 + \mathcal{LV}_1 -R(\mathcal{V}_1) = O(e^{-2e_0t}). 
\end{equation}

Suppose now that $\mathcal{V}_1$, $\mathcal{V}_2$, $\cdots,$ $\mathcal{V}_l$ are defined. Write
\begin{equation}\label{def_eps_l}
    \epsilon_l = \partial_t \mathcal{V}_l + \mathcal{LV}_l -  iR(\mathcal{V}_l)
\end{equation}
and note that
\begin{equation}
    \partial_t \mathcal{V}_l = -\sum_{i=1}^l j e_0 e^{-je_0t}Z_j,
\end{equation}
which allows us to write
\begin{equation}\label{epsilon_expansion}
    \epsilon_l = \sum_{i=1}^l e^{-je_0t}(-j e_0 Z_j+\mathcal{L}Z_j)-R(\mathcal{V}_l).
\end{equation}
Since $Z_j \in \mathcal{S}(\mathbb R^3)$ for all $j\leq l$, and from the explicit expression of $R$ one can rewrite the last equation as 
\begin{equation}
    \epsilon_l(x,t) = \sum_{i=1}^{l+1} e^{-je_0t}F_j(x) + O(e^{-e_0(l+2)t}),
\end{equation}
with $F_j \in \mathcal{S}(\mathbb R^3)$ for all $j$. Since $\epsilon_l = O(e^{-(l+1)e_0t})$, by the induction hypothesis, we conclude that $F_j = 0$ for $j \leq k$, allowing us to write
\begin{equation}
    \epsilon_l(x,t) = e^{-(l+1)e_0t}F_{l+1}(x) + O(e^{-e_0(l+2)t}).
\end{equation}
Define now $Z_{l+1} = -(\mathcal{L}-(l+1)e_0)^{-1}F_{l+1} \in \mathcal{S}(\mathbb R ^3) \times \mathcal{S}(\Real^3)$, which can be done since $(l+1)e_0$ does not belong to the spectrum of $\mathcal{L}$. It remains to estimate
\begin{equation}
    \epsilon_{l+1} = \epsilon_l - e^{-(l+1)e_0}F_{l+1}-i(R(\mathcal{V}_{l+1})-R(\mathcal{V}_l)).
\end{equation}
But this is clear, since $\epsilon_l - e^{-(l+1)e_0}F_{l+1} = O(e^{-(l+2)e_0t})$ and the explicit expression of $R$ gives $R(\mathcal{V}_{l+1})-R(\mathcal{V}_l) = O(e^{-(l+2)e_0t})$.
\end{proof}

\subsection{Fixed-point argument near an approximate solutions}\label{def_l_0}

We now want to find an exact solution to \eqref{linearized_eq} close to $\mathcal{V}_l$. Given a ground state $Q = (\phi,\psi)$ and recalling \eqref{linearized_eq_2} and \eqref{def_eps_l}, we want to solve
\begin{equation}
    i\partial_t w + \Delta w - w + L(w) + i[R(\mathcal{V}_l+w)-R(\mathcal{V}_l)]-i\epsilon_l=0.
\end{equation}

We now show:

\begin{prop}\label{sub_exist_UA} Given any ground state $Q=(\phi,\psi)$, there exists $l_0 > 0$ such that for any $l \geq l_0$, there exists $t_l \geq 0$ and a solution $U^A$ to \eqref{sys_NLS} such that for $t \geq t_l$, we have
\begin{equation}\label{sub_bound_UA}
\|\langle\nabla\rangle(U^A-e^{it}Q-e^{it}\mathcal{V}_{l}^A)\|_{S'(L^2,[t,+\infty))} \leq e^{-(l+\frac{1}{2})e_0 t}.
\end{equation}

Furthermore, $U^A$ is the unique solution to the NLS system \eqref{sys_NLS} satisfying \eqref{sub_bound_UA} for large $t$. Finally, $U^A$ is independent of $l$ and satisfies for large $t$,
\begin{equation}\label{sub_bound_UA2}
	\|U^A(t) - e^{it}Q - Ae^{-e_0t+it}\mathcal{Y}_+\|_{H^1} \leq e^{-2e_0 t}.
\end{equation}
\end{prop}

\begin{proof}

Given a ground state $Q = (\phi,\psi)$, we define
\begin{equation}
    \mathcal{M}(w)(t) = i \int_t^{+\infty} e^{i(t-s)(\Delta-1)}\left\{L(w) + i[R(\mathcal{V}_l+w)-R(\mathcal{V}_l)] + i\epsilon_l \right\} \, ds
\end{equation}

for $w$ in the (complete) metric space $B = B(l,t_l)$ defined as
\begin{align}
    B &:= \{ w \in E: \|w\|_E \leq 1\},\\
    E &:= \{ w \in S(L^2,[t_l+\infty))\times S(L^2,[t_l+\infty)) : \|w\|_{E}<+\infty\},\\
    \|w\|_E &:= \sup_{t \geq t_l} e^{\left(l+\frac{1}{2}\right)t}\|\langle \nabla \rangle w\|_{S(L^2,[t+\infty))\times S(L^2,[t+\infty))},
\end{align}

equipped with the metric
\begin{equation}
    \rho(w,\tilde{w}) = 
    \sup_{t \geq t_l} e^{\left(l+\frac{1}{2}\right)e_0t}
    \|w - \tilde{w}\|_E.
\end{equation}

We will show now that $\mathcal{M}$ is a contraction on $B$, if $l$ and $t_l$ are chosen to be large.

Indeed, if $w, \tilde{w} \in B$, by Strichatz's inequality,
\begin{align}
    \|\langle \nabla \rangle \mathcal{M}(w)\|_{S(L^2,[t,+\infty))\times S(L^2,[t,+\infty))}
    &\lesssim 
    \|\langle \nabla \rangle L(w)\|_{S'(L^2,[t,+\infty))\times S'(L^2,[t,+\infty))}
    \\&\quad+
    \|\langle \nabla \rangle (R(\mathcal{V}_l+w)-R(\mathcal{V}_l))\|_{S'(L^2,[t,+\infty))\times S'(L^2,[t,+\infty))}
    \\&\quad+
    \|\langle \nabla \rangle \epsilon_l\|_{S'(L^2,[t,+\infty))\times S'(L^2,[t,+\infty))} \end{align}
and 
\begin{align}
    \|\langle \nabla \rangle [\mathcal{M}(w)-\mathcal{M}(\tilde{w})]\|_{S(L^2,[t,+\infty))\times S(L^2,[t,+\infty))}
    &\lesssim 
    \|\langle \nabla \rangle L(w-\tilde{w})\|_{S'(L^2,[t,+\infty))\times S'(L^2,[t,+\infty))}
    \\&\hspace{-0.8cm}+
    \|\langle \nabla \rangle (R(\mathcal{V}_l+w)-R(\mathcal{V}_l+\tilde{w}))\|_{S'(L^2,[t,+\infty))\times S'(L^2,[t,+\infty))}.
\end{align}
By Lemma \ref{prel_estimates}, for $0< \tau < 1$, if $I^{\tau}_{j}=[t+j\tau,t+(j+1)\tau]$:
\begin{align}
     \|\langle \nabla \rangle L(w)\|_{S'(L^2,[t,+\infty))\times S'(L^2,[t,+\infty))} 
     &\leq \sum_{j=0}^\infty  \|\langle \nabla \rangle L(w)\|_{S'(L^2,I^{\tau}_{j})\times S'(L^2,I^{\tau}_{j})}
     \\&\lesssim  \sum_{j=0}^\infty \tau^{\frac{1}{2}} \| \langle \nabla \rangle w\|_{S(L^2,I^{\tau}_{j})\times S(L^2,I^{\tau}_{j})}\\
     &\lesssim \tau^{\frac{1}{2}} \sum_{j=0}^\infty  e^{-\left(j+\frac{1}{2}\right)e_0(t+j\tau)}\\
     &=    \tau^{\frac{1}{2}} \|w\|_E e^{-\left(l+\frac{1}{2}\right)e_0t}  \|w\|_E\sum_{j=0}^\infty  e^{-j\left(l+\frac{1}{2}\right)e_0\tau}\\
     &= e^{-\left(l+\frac{1}{2}\right)e_0t} 
     \frac{\tau^{\frac{1}{2}}}{1-e^{-\left(l+\frac{1}{2}\right)e_0\tau}}\|w\|_E.
\end{align}
If $l_0 := \frac{\ln 2}{e_0} \frac{1}{\tau}-\frac{1}{2}$, we have, for all $l > l_0(\tau)$,
\begin{equation}
    \|\langle \nabla \rangle L(w)\|_{S'(L^2,[t,+\infty))\times S'(L^2,[t,+\infty))} \lesssim \tau^\frac{1}{2} e^{-\left(l+\frac{1}{2}\right)e_0t}\|w\|_E.
\end{equation}
Similarly, abbreviating $S(L^2,I_j^{\tau})\times S(L^2,I_j^{\tau}) = S(L^2,I_j^{\tau})^2$ and $S'(L^2,I_j^{\tau})\times S'(L^2,I_j^{\tau}) = S'(L^2,I_j^{\tau})^2$,
\begin{align}
     &\|\langle \nabla \rangle (R(\mathcal{V}_l+w)-R(\mathcal{V}_l+\tilde{w}))\|_{S'(L^2,I_j^{\tau})^2}\\ &\quad\lesssim 
     \left[\|\langle\nabla\rangle \mathcal{V}_l\|_{S(L^2,I_j^{\tau})^2}
    +
    \|\langle\nabla\rangle w\|_{S(L^2,I_j^{\tau})^2}
    +
    \|\langle\nabla\rangle \tilde{w}\|_{S(L^2,I_j^{\tau})^2}\right.\\
    &\quad\quad+\left.\|\langle\nabla\rangle \mathcal{V}_l\|_{S(L^2,I_j^{\tau})^2}^2
    +
    \|\langle\nabla\rangle w\|_{S(L^2,I_j^{\tau})^2}^2
    +
    \|\langle\nabla\rangle \tilde{w}\|_{S(L^2,I_j^{\tau})^2}^2\right]
    \|\langle\nabla\rangle(w-\tilde{w})\|_{S(L^2,I_j^{\tau})^2}\\
    &\quad\lesssim_l\left[e^{-e_0t} + e^{-\left(l+\frac{1}{2}\right)e_0t}(\|w\|_E+\|\tilde{w}\|_E) \right]e^{-\left(l+\frac{1}{2}\right)e_0(t+j\tau)}\|w-\tilde{w}\|_E\\
    &\quad\lesssim_l e^{-\left(l+1\right)e_0t}\|w-\tilde{w}\|_E \, e^{-j\left(l+\frac{1}{2}\right)\tau},
\end{align}
we have
\begin{equation}
    \|\langle \nabla \rangle (R(\mathcal{V}_l+w)-R(\mathcal{V}_l+\tilde{w}))\|_{S'(L^2,[t,+\infty))^2} \lesssim_l e^{-\left(l+1\right)e_0t}\|w-\tilde{w}\|_E.
\end{equation}
Finally, by construction,
\begin{equation}
    \|\epsilon_l\|_{S'(L^2,[t,+\infty))^2} \lesssim_l e^{-(l+1)e_0t},
\end{equation}
which proves, for $t \geq t_l$, with $t_l>0$ large enough,
\begin{equation}
    \|\mathcal{M}(w)\|_{E} \leq \left[C \tau^{\frac{1}{2}} + C_l e^{-\left(l+\frac{1}{2}\right)e_0 t_l}\right] \leq 1,
\end{equation}
and
\begin{equation}
    \|\mathcal{M}(w)-\mathcal{M}(\tilde{w})\|_{E} \leq \left[C \tau^{\frac{1}{2}} + C_l e^{-\left(l+\frac{1}{2}\right)e_0 t_l}\right]\|w-\tilde{w}\|_E \leq \frac{1}{2}\|w-\tilde{w}\|_E.
\end{equation}

(Recall the order of choices: one first chooses a small universal $\tau>0$, then a large $l>l_0(\tau)$ and finally a large $t_l(l)$).
\end{proof}
\begin{re} We remark that the uniqueness condition still holds if, given $l>l_0$, one chooses a larger $\tilde{t}_l > t_l$. Moreover, the function $U_l := e^{it}(Q+ \mathcal{V}_l + w_l)$ is independent of $l$, since given $l'>l>l_0$, $t_{l'} > t_l$ and two solutions $w_l$ and $w_{l'}$,  respectively, one has two solutions on $B(l,t_{l'})$, namely $w_l$ restricted to $t \in [t_{l'},+\infty)$ and $\tilde{w}_l := \mathcal{V}_{l'}- \mathcal{V}_{l} + w_{l'}$, which, by uniqueness of \eqref{linearized_eq_2}, must coincide on $[t_l',+\infty)$ and by uniqueness of solutions to \eqref{sys_NLS}, must also coincide on $[t_l,+\infty)$.

\end{re}

\section{Behavior of solutions at the mass-energy threshold}\label{Sec7}
Having shown the existence of different solutions at the threshold level, we classify the possible behaviors in this setting. As in the case $\mathcal{ME}(u_0,v_0)<1$, we have different behaviors depending on whether the mass-kinetic energy is high ($\mathcal{MK}(u_0,v_0)>1)$ or low ($\mathcal{MK}(u_0,v_0)<1$). The case $\mathcal{MK}(u_0,v_0) = 1$ implies that $(u,v)$ is, up to symmetries, the standing wave $(e^{it}\phi,e^{it} \psi)$.
\subsection{High kinetic energy}

We treat the solutions such that $\mathcal{ME}(u_0,v_0)=1$ and $\mathcal{MK}(u_0,v_0)>1$ first. Up to a constant rescaling, we assume $M(u_0,v_0) = M(\phi,\psi)$, $E(u_0,v_0) = E(\phi,\psi)$ and $K(u_0,v_0) > K(\phi,\psi)$. By uniqueness of the flow, this implies $K(u(t),v(t)) > K(\phi,\psi)$ for all $t$ in the maximal interval of existence of $(u,v)$. Recalling the definition of $\delta$ \eqref{def_delta}, in this case we have
\begin{equation}
    \delta(t) = K(u(t),v(t))-K(\phi,\psi).
\end{equation}

This section is devoted to proving the following Lemma.

\begin{lemma}\label{conv_ground_high} Let $(u,v)$ be a solution to \eqref{sys_NLS} such that $M(u_0,v_0) = M(\phi,\psi)$, $E(u_0,v_0) = E(\phi,\psi)$ and $K(u_0,v_0)>K(\phi,\psi)$. Assume, in addition, that either $u_0$ is radial or has finite variance. Then either $(u,v)$ blows up in finite positive time, or there exist $x_0 \in \mathbb{R}^3$, $\theta_0 \in \mathbb{R}/2\pi \mathbb{Z}$, $(\phi_0,\psi_0) \in \mathcal{G}$ and $c>0$ such that 
\begin{equation}
    \|(u(t),v(t)) - (e^{i(\theta_0+t)}\phi_0(\cdot+x_0),e^{i(\theta_0+t)}\psi_0(\cdot+x_0))\|_{H^1 \times H^1} \lesssim e^{-ct}.
\end{equation}

Moreover, $(u,v)$ blows up in finite negative time in $H^1 \times H^1$.
\end{lemma}

\subsubsection{Radial case}

We define $a$ to be a smooth, radial function such that
\begin{equation}
    a(x) = \begin{cases}
    |x|^2, & |x| < R\\
    0, & |x| > 3R
    \end{cases} 
\end{equation}

and that $\partial^2_r a(x) \leq 2$ for all $x\neq 0$, where $\partial_r a(x) = \frac{x \cdot \nabla a(x)}{|x|}$ is the radial derivative and $R \geq 1$.

We define the virial quantity
\begin{equation}
    V_R(t) = \int a \left[|u(t)|^2 + |v(t)|^2\right], 
\end{equation}

for which we have the identities
\begin{equation}
    V_R'(t) = 2 \Im \int \nabla a \cdot (\nabla u(t) \overline{u}(t)+\nabla v(t) \overline{v}(t))
\end{equation}
and, in the case $E(u_0,v_0) = E(\phi,\psi)$,
\begin{equation}
    V_R''(t) = -4\delta(t) + A_R(u(t),v(t)),
\end{equation}
where 
\begin{equation}\label{def_AR}
    A_R(u,v) = 2\int(\partial_r^2 a -2) \left[|\nabla u|^2+|\nabla v|^2\right] -\int(\Delta^2 a)\left[| u|^2+| v|^2\right] - \int (\Delta a-6)\left[|u|^4 + 2 \beta |uv|^2 + |v|^4\right].
\end{equation}

We now give different bounds for $A_R$, depending on $\delta$. If $\delta\geq \delta_0$, since $\partial_r^2 a -2 \leq 0$, $|\Delta^2 a| \leq \frac{1}{R^2}$ and $|\nabla a -6| \lesssim 1$, one has, by the well known radial Strauss lemma,
\begin{equation}
    A_R(u,v) \leq \frac{1}{R^2}M(u,v) + \frac{1}{R^2}M(u,v)^\frac{3}{2}\left[K(\phi,\psi) + \delta\right]^\frac{1}{2} \leq \delta,
\end{equation}

if one chooses $R = R(M(u_0,v_0),\delta_0)$ large. Now, for $\delta<\delta_0$, we recall that $A_R(e^{it}\phi,e^{it}\psi) = 0$, since the corresponding $V_R$ is constant. Therefore, writing $(u,v) = (e^{it}(\phi+h),e^{it}(\psi+k))$,
\begin{align}
       |A_R(u(t),v(t))| &= |A_R(h(t)+\phi,k(t)+\psi) - A_R(\phi,\psi)| 
       \\&\lesssim \int_{|x| \geq R} |\nabla \phi| |\nabla h(t)| +
    |\nabla \psi| |\nabla k(t)|
    + |\nabla h(t)|^2 + |\nabla k(t)|^2\\
    &\quad+\int_{|x|\geq R} |\phi||h(t)|+|\psi||k(t)|+
    |h(t)|^2+|k(t)|^2 \\ &\quad+\int_{|x|\geq R}|\phi|^3|h(t)| + |\psi|^3|k(t)| + |h(t)|^4 + |k(t)|^4 \\
    &\lesssim (e^{-R}+\delta_0^\frac{1}{2})\delta(t),
\end{align}

since $(\phi,\psi)$ and their derivatives decay as $e^{-|x|}$. Choosing a possibly larger $R$ (but still independent on time), and a possibly smaller $\delta_0$, we conclude, in any case, that
\begin{equation}
    A_R(u(t),v(t)) \leq 2\delta(t).
\end{equation}

Therefore, $V''_R(t) \leq - 2\delta(t)$, if one chooses $R$ large. With this bound on $V_R''$, we are able to show:

\begin{lemma} Let $(u,v)$ be a radial solution to \eqref{sys_NLS} such that $M(u_0,v_0) = M(\phi,\psi)$, $E(u_0,v_0) = E(\phi,\psi)$ and $K(u_0,v_0)>K(\phi,\psi)$. If $(u(t),v(t))$ is defined for all $t>0$, then $(u_0,v_0)$ has finite variance.
\end{lemma}
\begin{proof}

We first recall that $ V_R(t)>0$ and $V_R''(t) \leq -\delta(t) < 0$ for all $t$, for sufficient large $R$. This means that $V'_R$ is decerasing and positive (otherwise $V_R(t)$ would become negative in finite time). Therefore, the limit 
\begin{equation}
    \lim_{t\to +\infty} V'_R(t)
\end{equation}
exists, and so the integral
\begin{equation}
    \int_t^{+\infty} V''_R(s) \,ds
\end{equation}
is convergent, which implies
\begin{equation}
     \int_0^{+\infty} \delta(s) \,ds < +\infty.
\end{equation}

This alone allows us to use Lemma \ref{const_modul} to conclude that there exist $\theta_0 \in \mathbb{R}/2\pi \mathbb{Z}$ such that 
\begin{equation}\label{slow_bound}
    \|(u(t),v(t)) - (e^{i(\theta_0+t)}\phi,e^{i(\theta_0+t)}\psi)\|_{H^1 \times H^1} \to 0,
\end{equation}
as $t \to +\infty$. However, we can also get quantitative bounds, which will be important later. To do this, we first note that \eqref{slow_bound} implies that the variance of $(u,v)$ is always finite, since $V_R$ is non-decreasing and, as $t \to +\infty$,
\begin{equation}
    \int_{|x|\leq R} |x|^2 \left[|u_0|^2+|v_0|^2\right] \leq  \int a \left[|u(t)|^2+|v(t)|^2\right] \to  \int a\left[\phi^2+\psi^2\right] 
\end{equation}

for all $R>0$ sufficiently large. Since the variance of $(\phi,\psi)$ is finite, due to the exponential decay, one can make $R\to +\infty$ in the last inequality to conclude that 
\begin{equation}
 \int |x|^2 \left[|u_0|^2+|v_0|^2\right] \leq \int |x|^2\left[\phi^2+\psi^2\right],
\end{equation}
which gives the desired
\end{proof}

We have then reduced the problem to treating finite-variance solutions, which is done below.

\subsubsection{Finite-variance case}

Here, there is no need to truncate the variance, so we define
\begin{equation}
    V(t) = \int |x|^2\left[|u(t)|^2 + |v(t)|^2\right], 
\end{equation}
for which we have the identities
\begin{equation}
    V'(t) = 4 \Im \int x \cdot (\nabla u(t) \overline{u}(t)+\nabla v(t) \overline{v}(t))
\end{equation}
and, in the case $E(u_0,v_0) = E(\phi,\psi)$,
\begin{equation}
    V''(t) = -4\delta(t).
\end{equation}

As in the radial case, we have that $V'$ is positive and decreasing, and again by Lemma \ref{const_mod_param}, $V$ is uniformly bounded above by the ground state variance. Now, to get the desired exponential decay on time, we make use of the following Cauchy-Schwarz-type inequality

\begin{lemma}\label{Banica} Let $a$ be smooth and $(f,g) \in H^1 \times H^1$. If $|\nabla a|f$ and $|\nabla a|g$ belong to $L^2$, and assuming
\begin{equation}
    M(f,g) = M(\phi,\psi), \quad E(f,g) = E(\phi,\psi),
\end{equation}
then
\begin{equation}
    \left(\Im \int \nabla a \cdot (\nabla f \overline{f}+\nabla g \overline{g}) \right)^2 \lesssim \delta^2(f,g) \int |\nabla a|^2(|f|^2+|g|^2).
\end{equation}
\end{lemma}
\begin{proof}
By using the sharp Gagliardo-Nirenberg inequality on $(e^{i\lambda a}f, e^{i\lambda a}g)$, we get, for all $\lambda \in \mathbb R$,
\begin{equation}
    \lambda^2 \int |\nabla a|^2(|f|^2+|g|^2) + 2 \lambda \Im \int \nabla a \cdot (\nabla f \overline{f}+\nabla g \overline{g}) + 
    K(f,g)-\frac{|P(f,g)|^{\frac{2}{3}}}{c_{GN}^\frac{2}{3}M(f,g)^{\frac{1}{3}}} \geq 0.
\end{equation}
The result then follows by noting that, if $M(f,g) = M(\phi,\psi)$ and $E(f,g) = E(\phi,\psi)$, then
\begin{equation}
    K(f,g)-\frac{|P(f,g)|^{\frac{2}{3}}}{c_{GN}^\frac{2}{3}M(f,g)^{\frac{1}{3}}} = K(\phi,\psi)\left[1+\frac{\delta(f,g)}{K(\phi,\psi)}-\left|1+ \frac{3\delta(f,g)}{2K(\phi,\psi)}\right|^{\frac{2}{3}}\right] \lesssim \delta^2(f,g).
\end{equation}
\end{proof}
In particular, choosing $a(x) = |x|^2$, we obtain
\begin{equation}
    V'(t) \lesssim  \sqrt{V(t)}\, \delta(t)\lesssim - \sqrt{\int |x|^2\left[\phi^2+\psi^2\right]}\, V''(t). 
\end{equation}
In other words, there exists $c>0$ such that
\begin{equation}
    V''(t) \leq -c V'(t),
\end{equation}
which immediately implies
\begin{equation}
    \int_t^{+\infty} \delta(s) \, ds = \frac{1}{2} V'(t) \leq \frac{1}{2}V'(0) e^{-ct}
\end{equation}

In view of Lemma \ref{const_modul}, we have just proved Lemma \ref{conv_ground_high}, up to the blow up in finite negative time. But this is a consequence of $V'(0)>0$, $V''\leq 0$ and the time-reversal symmetry.

\subsection{Low kinetic energy}

We now treat the case $\mathcal{ME}(u_0,v_0) = 1$ and $\mathcal{MK}(u_0,v_0)< 1$. Again, up to scaling, we assume $M(u_0,v_0) = M(\phi,\psi)$, $E(u_0,v_0) = E(\phi,\psi)$ and $K(u_0,v_0) < K(\phi,\psi)$. In this case, $\delta(t)$ is given by
\begin{equation}
    \delta(t) = K(\phi,\psi) - K(u(t),v(t)).
\end{equation}

We then prove a result similar to Lemma \ref{conv_ground_high}:

\begin{lemma}\label{conv_ground_low} Let $(u,v)$ be a solution to \eqref{sys_NLS} such that $M(u_0,v_0) = M(\phi,\psi)$, $E(u_0,v_0) = E(\phi,\psi)$ and $K(u_0,v_0)<K(\phi,\psi)$. Then either $(u,v)$ scatters in positive time, or there exist $x_0 \in \mathbb{R}^3$, $\theta_0 \in \mathbb{R}/2\pi \mathbb{Z}$, $(\phi_0,\psi_0) \in \mathcal{G}$ and $c>0$ such that 
\begin{equation}
    \|(u(t),v(t)) - (e^{i(\theta_0+t)}\phi_0(\cdot+x_0),e^{i(\theta_0+t)}\psi_0(\cdot+x_0))\|_{H^1 \times H^1} \lesssim e^{-ct}.
\end{equation}

Moreover, $(u,v)$ is defined for all $t \in \mathbb R$ and scatters in negative time in $H^1 \times H^1$.
\end{lemma}

The difference here is that we cannot rely on some ``trapping'' given by the variance, but instead rely on compactness, given by the scattering norm being infinite. The following two lemmata are essentially the same as in \cite{DR_Thre}, to where we refer the reader for a proof.

\begin{lemma}\label{compactness_lemma} Let $(u,v)$ be a solution to \eqref{sys_NLS} with $M(u_0,v_0) = M(\phi,\psi)$ and $E(u_0,v_0) = E(\phi,\psi)$, defined for all $t\geq0$ and which does not scatter in positive time. Then, there exists a continuous function $x: [0,+\infty) \to \mathbb R^3$ which coincides with the translation parameter defined in \eqref{mod_decomposition} on the set $\{t \,|\, \delta(t)<\delta_0\}$ and
such that the set
\begin{equation}
    \left\{ (u(x-x(t),t),v(x-x(t),t))\,|\, t \in [0,+\infty) \right\}
\end{equation}
is precompact in $H^1\times H^1$. Moreover, if $K(u_0,v_0) < K(\phi,\psi)$, we have
\begin{equation}
    \frac{x(t)}{t} \to 0, \quad \text{as } t \to +\infty.
\end{equation}
\end{lemma}

\begin{lemma}\label{control_x} There exists $C_0>0$ such that if {$T_1 \geq T_0+1$ \color{red}}, then
\begin{equation}
    |x(T_1)-x(T_0)| \leq C_0 \int_{T_0}^{T_1} \delta(s) \, ds.
\end{equation}
\end{lemma}

Even though we do not have radiality here, we work with a truncated variance to make use of compactness. For $R \geq 1$, let $a$ be a smooth, positive, radial function such that
\begin{equation}
    a(x) = \begin{cases}
    |x|^2, & |x| < R\\
    2, & |x| > 3R
    \end{cases} 
\end{equation}
and that $|\nabla a|^2 \lesssim a$. As usual, define $V_R(t) = \displaystyle\int a \left[|u|^2+|v|^2\right]$ and note that $|V_R(t)| \lesssim R^2$. We also have, by the virial identities,
\begin{equation}\label{second_virial_below}
    V_R''(t) = 2 \delta(t) + A_R(u(t),v(t)),
\end{equation}
where $A_R$ is defined in \eqref{def_AR}.

\begin{lemma}\label{bound_delta_below} There exists $C\geq 1$ such that, for any $0\leq T_0 \leq t \leq T_1$, one has
\begin{align}
    |A_R(u(t),v(t))|\leq \delta(t),\label{bound_AR_below}\\
    |V'_R(t)|\lesssim  R \,\delta(t)\label{bound_FprimeR_below},
\end{align}
where $R = R(T_0,T_1,\eta) = \displaystyle\sup_{T_0\leq t \leq T_1}|x(t)| + C$.
\end{lemma} 
\begin{proof}
As in the previous section, we have, for $\delta(t) < \delta_0$, upon writing 
\begin{equation}
    (u(x,t),v(x,t)) = \left(e^{i\theta(t)}[\phi(x-x(t))+h(x-x(t),t)],e^{i\theta(t)}[\psi(x-x(t))+k(x-x(t),t)]\right).
\end{equation} 
we have, for any (fixed) $t \in [a,b]$,
\begin{align}
       |A_R(u(t),v(t))| &= |A_{R}(h(\cdot-x(t),t)+\phi,k(\cdot-x(t),t)+\psi) - A_{R}(\phi(\cdot-x(t)),\psi(\cdot-x(t)))| 
       \\&\lesssim \int_{|x+x(t)| \geq R} |\nabla \phi| |\nabla h(t)| +
    |\nabla \psi| |\nabla k(t)|
    + |\nabla h(t)|^2 + |\nabla k(t)|^2\\
    &\quad+\int_{|x+x(t)|\geq R} |\phi||h(t)|+|\psi||k(t)|+
    |h(t)|^2+|k(t)|^2 \\ &\quad+\int_{|x+x(t)|\geq R}|\phi|^3|h(t)| + |\psi|^3|k(t)| + |h(t)|^4 + |k(t)|^4 \\
    &\lesssim (e^{-(R-x(t))}+\delta_0^\frac{1}{2})\delta(t)\\
    &\lesssim (e^{-C}+\delta_0^\frac{1}{2})\delta(t).
\end{align}

Choosing a smaller $\delta_0$, if necessary, and a large enough $C\geq 1$, equation \eqref{bound_AR_below} is proved in the case $\delta(t)<\delta_0$. If $\delta(t) \geq \delta_0$ and $t \in [a,b]$, we write
\begin{align}
    |A_R(u(t),v(t))| &\lesssim \int_{|x|\geq R } \left[|\nabla u|^2+|\nabla v|^2 +| u|^2+| v|^2 + |u|^4 +|v|^4\right]\\
    &\lesssim \int_{|x-x(t)|\geq R-|x(t)| } \left[|\nabla u|^2+|\nabla v|^2 +| u|^2+|u|^4 +|v|^4\right]\\
    &\lesssim \int_{|x-x(t)|\geq C } \left[|\nabla u|^2+|\nabla v|^2 +| u|^2+ |u|^4 +|v|^4\right].\\
\end{align}
Thus, by possibly increasing $C\geq 1$, we have, by compactness, $|A_R(u(t),v(t))|\leq \delta_0 \leq \delta(t)$.

To get \eqref{bound_FprimeR_below}, we make use of Lemma \ref{Banica} and the definition of $a$ to write
\begin{equation}
    |V_R'(t)| \lesssim \sqrt{V_R(t)} \delta(t) \lesssim R \delta(t),
\end{equation}
which proves the desired.
\end{proof}

\begin{cor}\label{virial_delta} There exists $C_1>0$ such that, for any $0\leq T_0\leq t \leq  T_1$, one has
\begin{equation}\label{estimate_virial_delta}
\int_{T_0}^{T_1} \delta(s) \, ds \leq C_1 \left[1+\sup_{T_0\leq t \leq T_1}|x(t)|\right](\delta(T_1)+\delta(T_0)) =  \,o_{T_0}(T_1),
\end{equation}
where $o_{T_0}(T_1)/T_1 \to 0$ as $T_1 \to +\infty$ and $T_0$ is fixed. In particular, there exists a sequence $\{t_n\}_{n = 0}^{\infty}$ with $t_{n+1}\geq t_n+1$ and such that $\delta(t_n) \leq \displaystyle\frac{1}{2^{n+2}C_0C_1}$ for all $n$, where $C_0$ is defined in Lemma \ref{control_x}.
\end{cor}
\begin{proof}
By equations \eqref{second_virial_below}, \eqref{bound_AR_below} and \eqref{bound_FprimeR_below}, one has
\begin{equation}
    \int_{T_0}^{T_1} \delta(s) \, ds \leq  \int_{T_0}^{T_1} V_R''(s) \, ds \leq |V_R'(T_1)|+|V_R'(T_0)| \lesssim \left[1+\sup_{T_0\leq t \leq T_1}|x(t)|\right](\delta(T_1)+\delta(T_0)).
\end{equation}
Recalling that $\delta(t) \leq K(\phi,\psi)$ for all $t$ and that $x(t) = o_1(t)$ (by Lemma \ref{compactness_lemma}), we get \eqref{estimate_virial_delta}. The existence of the sequence $\{t_n\}$ comes from the mean value theorem. 
\end{proof}

\begin{cor}\label{x_bounded} We have
\begin{equation}
\sup_{t \in [0,+\infty)}|x(t)| \lesssim 1 + \sup_{t_0\leq s \leq t_1}|x(t)| < +\infty
\end{equation}
\end{cor}
\begin{proof}
Let $t^*_n$ be such that $|x(t^*_n)| = \displaystyle\sup_{t_1 \leq t \leq t_n} |x(t)|$, where $\{t_n\}$ is given in Lemma \ref{virial_delta}. By Lemmas \ref{control_x} and \ref{virial_delta}, we can write
\begin{align}
    |x(t^*_n)| &\leq |x(t_0)| + C_0 C_1 \left[1+\sup_{t_0 \leq t \leq t_1}|x(t)|+|x(t^*_n)|\right]\frac{1}{2C_0C_1}\\
    &\leq |x(t_0)| + \frac{1+\sup_{t_0 \leq t \leq t_1}|x(t)|}{2} + \frac{|x(t^*_n)|}{2}.
\end{align}
\end{proof}

Summing up all of the results of this section, we get the desired bound on $\delta$, which implies Lemma \ref{conv_ground_low} in view of Lemma \ref{const_modul}, up to the scattering for negative time:
\begin{cor}
\begin{equation}
\int_t^{+\infty}\delta(s) \, ds \lesssim e^{-ct}.    
\end{equation}
\end{cor}
\begin{proof}
By Corollaries \ref{virial_delta} and \ref{x_bounded}, we have, for any $t\geq0$,
\begin{equation}
    \int_{t}^{t_n} \delta(s)\, ds \lesssim \delta(t)+\delta(t_n).
\end{equation}
The result follows by making $n \to +\infty$ and using Gronwall's lemma.
\end{proof}

The scattering part then follows from the following result.

\begin{cor} $(u,v)$ scatters for negative time in $H^1 \times H^1$
\end{cor}
\begin{proof} The boundedness of the $H^1 \times H^1$ norm implies global existence of the solution. If it does not scatter for negative time, we can apply the results of this section to the time-reversed solution $t\mapsto(\overline u(-t),\overline{v}(-t))$ to conclude that
\begin{equation}
    \lim_{t\to \pm \infty}\delta(t) = 0
\end{equation}
and that, for any $t \in \mathbb R$, 
\begin{equation}
    \int_{-t}^{t}\delta(s)\, ds \lesssim \delta(t)+\delta(-t).
\end{equation}
The last two equations then imply $\delta \equiv 0$.
\end{proof}

\section{Exponentially decaying solutions to the linearized equation}\label{Sec8}

Given any ground state $Q=(\phi,\psi)$, we now study solutions to the equation
\begin{equation}\label{eq_exp_decay}
    \partial_t (h,k) + \mathcal{L}(h,k) = g,
\end{equation}
where 
\begin{align}\label{eq_exp_decay_2}
\|(h(t),k(t))\|_{H^1 \times H^1} \lesssim e^{-\gamma_1 t} \text{ and } \|g(t)\|_{H^1 \times H^1} \lesssim e^{-\gamma_2 t}\text{ , with }\gamma_2>\gamma_1>0.    
\end{align}

We first normalize the functions in the null space of $\mathcal{L}$ by enumerating $\ker(\mathcal{L}) = \{\tilde Q_j\}_{j=1}^{\dim(\ker(\mathcal{L}))}$ and writing for all $j$
\begin{equation}
    Q_j = \frac{1}{\|\tilde{Q}_j\|_{L^2 \times L^2}} \tilde Q_j.
\end{equation}

We also renormalize $\mathcal{Y}_{\pm}$, as to satisfy
\begin{equation}
    B(\mathcal{Y}_+,\mathcal{Y}_-) = 1.
\end{equation}
Write 
\begin{equation}
    (h,k) = \alpha_+(t) \mathcal{Y}_+ +\alpha_-(t) \mathcal{Y}_- + \sum_{j} \beta_j(t) Q_j + w(t),
\end{equation}
with $w(t) \in \tilde{G}^\perp$. Note that
\begin{align}
    \alpha_+(t) = B((h,k),\mathcal{Y}_-)\\
    \alpha_-(t) = B((h,k),\mathcal{Y}_+)\\
    {\beta_j(t) = ((h,k),Q_j)}\\
    \Phi(h,k) = \Phi(w) + \alpha_+ \alpha_-
\end{align}
In particular, $|\alpha_+(t)|+|\alpha_-(t)| +\sum_j |\beta_j(t)| + \|w(t)\|_{H^1 \times H^1} \lesssim e^{-\gamma_1t}$.

\subsection{Differential equations for the modulation parameters}
The following result follows from direct differentiation.
\begin{prop}
We have
\begin{align}
    \frac{d}{dt}(e^{e_0t}\alpha_+) = e^{e_0t}B(g,\mathcal{Y}_-)\\
    \frac{d}{dt}(e^{-e_0t}\alpha_-) = e^{-e_0t}B(g,\mathcal{Y}_+)\\
    {\frac{d}{dt}\beta_j = (g,Q_j)}, \quad\forall\, j\\
    \frac{d}{dt}\Phi(h,k) = 2B(g,(h,k))
\end{align}
In particular,
\begin{align}
    \left|\frac{d}{dt}(e^{e_0t}\alpha_+)\right| \lesssim e^{(e_0-\gamma_2)t}\\
    \left|\frac{d}{dt}(e^{-e_0t}\alpha_-)\right| \lesssim e^{-(e_0+\gamma_2)t}\\
    {\sum_{j}\left|\frac{d}{dt}\beta_j\right| \lesssim e^{-\gamma_2 t}}\\
    \left|\frac{d}{dt}\Phi(h,k)\right|| \lesssim e^{-(\gamma_1+\gamma_2)t}.
\end{align}
\end{prop}
\subsection{Self-improving decay}
The decay of $g$ can be used to improve the decay of $(h,k)$.
\begin{lemma} 
\begin{align}
    |\alpha_-(t)| \lesssim e^{-\gamma_2 t}\\
    \sum_{j}|\beta_j(t)| \lesssim e^{-\gamma_2 t}\\
    \|w(t)\|_{H^1 \times H^1} \lesssim e^{-\frac{\gamma_1+\gamma_2}{2}}.
\end{align}
\begin{proof}
The bounds follow from direct integration. 
\end{proof}
\end{lemma}

\begin{lemma}
There exists $A \in \mathbb{R}$ such that 
\begin{equation}
    |\alpha_+(t)-Ae^{-e_0t}| \lesssim e^{-\gamma_2^- t}.
\end{equation}
\end{lemma}
\begin{proof}
Write
\begin{equation}
    e^{e_0s}\alpha_+(s)-e^{e_0t} \alpha_+(t) = \int_t^s e^{e_0 \tau} B(g(\tau),\mathcal{Y}_-) \, d \tau.
\end{equation}
If $\gamma_2>e_0$, then the last integral converges, wich means $e^{e_0s}\alpha_+(s)$ has a limit $A$ as $s \to +\infty$. That implies
\begin{equation}
     |\alpha_+(t)-Ae^{-e_0t}| \lesssim e^{-e_0t}\int_t^{+\infty} e^{(e_0-\gamma_2) \tau} \, d \tau \lesssim e^{-\gamma_2t}.
\end{equation}
If $\gamma_2 \leq e_0$, then
\begin{equation}
    |e^{e_0t}\alpha_+(t)|  \lesssim |\alpha_+(0)| + e^{(e_0-\gamma_2^-)t},
\end{equation}
which means we can choose $A = 0$.
\end{proof}

Define now $(\tilde h, \tilde k) = (h,k) - Ae^{-e_0t}\mathcal{Y}_+$ and $\tilde{\alpha}_+ = \alpha_+ - Ae^{-e_0t}$. Note that $(\tilde h, \tilde k)$ satisfies the same differential equation as $(h,k)$. We then  have
\begin{equation}
    \|(\tilde h(t), \tilde k(t))\|_{H^1 \times H^1} \lesssim |\tilde{\alpha}_+(t)| + |\alpha_-(t)| + \sum_{j}|\beta_j(t)| + \|w(t)\|_{H^1 \times H^1} \lesssim e^{-\gamma_2^- t} + e^{-\frac{\gamma_1+\gamma_2}{2}} \lesssim e^{-\frac{\gamma_1+\gamma_2}{2}}.
\end{equation}
In view of the improved decay of $(\tilde h, \tilde k)$, we can replace $\gamma_1$ by $\frac{\gamma_1+\gamma_2}{2}$, and by repeated iteration, we conclude:
\begin{lemma}\label{lemma_improv_decay} If $(u,v)$ satisfies \eqref{eq_exp_decay} and \eqref{eq_exp_decay_2}, then there exists $A \in \mathbb R$ such that
\begin{equation}
    (h(t),k(t)) = Ae^{-e_0t}\mathcal{Y}_+ + O(e^{-\gamma_2^-t}) \quad \text{in }H^1 \times H^1.
\end{equation}
\end{lemma}

\subsection{Uniqueness}
\begin{prop}\label{prop_unique_decay}
Let $(u,v)$ be a solution to \eqref{sys_NLS} such that $\|(u(t),v(t)-(e^{it}\phi,e^{it}\psi)\|_{H^1 \times H^1} \lesssim e^{-ct}$ for some ground state $Q=(\phi,\psi)$ and some $c>0$. Then there exists $A \in \mathbb R$ such that $(u,v) = U^A$.
\end{prop}
\begin{proof}
Write $(u,v) = (e^{it}(\phi + h), e^{it}(\psi+k))$ and let $A \in \mathbb{R}$ be given by Lemma \ref{lemma_improv_decay}. Recall that $(h,k)$ satisfies
\begin{equation}
    \partial_t (h,k) + \mathcal{L}(h,k) = R(h,k),
\end{equation}
where 
\begin{equation}
    \|R(h,k)\|_{H^1 \times H^1} \lesssim \|(h,k)\|_{H^1 \times H^1}^2.
\end{equation}
Therefore, by Lemma \ref{lemma_improv_decay}, we have that $\|(h(t),k(t))\|_{H^1 \times H^1} \lesssim e^{-e_0t}$ and $\|R(h(t),k(t))\|_{H^1 \times H^1} \lesssim e^{-2e_0t}$. This lets us bootstrap the decay to conclude
\begin{equation}\label{bound_A_uniq}
    \|(h(t),k(t))-Ae^{-e_0t}\mathcal{Y}_+\|_{H^1 \times H^1} \lesssim e^{-2e_0^-t}.
\end{equation}

Consider the corresponding $U^A$ and write $U^A = (e^{it}(\phi + h^A), e^{it}(\psi+k^A))$. Recall that
\begin{equation}
    \|(h^A(t),k^A(t)) - Ae^{-e_0t}\mathcal{Y}_+\|_{H^1 \times H^1} \lesssim e^{-2e_0t}.
\end{equation}
Thus,
\begin{equation}
    \|(h(t),k(t)) - (h^A(t),k^A(t))\|_{H^1 \times H^1} \lesssim e^{-2e_0^-t}.
\end{equation}
We now note that $z := (h,k)-(h^A,k^A)$ satisfies
\begin{equation}
    \partial_t z + \mathcal{L}z = R(h,k)-R(h^A,k^A)
\end{equation}
and that $$\|R(h,k)-R(h^A,k^A)\|_{H^1 \times H^1} \lesssim (\|(h,k)\|_{H^1 \times H^1}+\|(h^A,k^A)\|_{H^1 \times H^1})\|(h,k)-(h^A,k^A)\|_{H^1 \times H^1}.$$
We conclude that, for any $\gamma>0$,
\begin{equation}
    \|(h,k)-(h^A,k^A)\|_{H^1 \times H^1}\lesssim_{\gamma}e^{-\gamma t}.
\end{equation}
In particular, choosing $\gamma>(l_0+1)e_0$, (where $l_0$ is defined in Section \ref{def_l_0}),{ and using Strichartz,} the uniqueness of $U^A$ implies $h = h^A$, i.e., $(u,v) = U^A$.
\end{proof}
The above result also allows us to narrow the number of solutions even further down. 
\begin{cor} Let $A \in \mathbb R$ and $Q = (\phi,\psi) \in \mathcal{G}$. We define $Q^+ = U^{+1}$ and $Q^- = U^{-1}$. Then, if $A>0$, there exists a time $T_A \in \mathbb R$ such that $U^A = e^{iT_A}Q^+(t-T_A)$. Conversely, if $A<0$, there exists a time $T_A \in \mathbb R$ such that $U^A = e^{iT_A}Q^-(t-T_A)$.
\end{cor}
\begin{proof}
Indeed, given $A>0$, define $T_A = - \frac{1}{e_0}\ln A$. Then
\begin{equation}
    e^{-iT_A}U^A(t+T_A) = e^{it}(Q + \mathcal{Y}_+e^{-e_0t})+O(e^{-2e_0t}).
\end{equation}

By using the same argument after \eqref{bound_A_uniq}, we conclude that $e^{-iT_A}U^A(t+T_A) = U^{+1}$. The proof for $A<0$ is completely analogous.
\end{proof}

We then sum the previous results up to obtain the main theorems of this paper:

\begin{proof}[Proof of Theorem \ref{sub_special}]
The existence part follows from Proposition \ref{sub_exist_UA}. The finite-variance and the blow-up part for $Q^+$ come from Lemma \ref{conv_ground_high} and the scattering backwards in time for $Q^-$ comes from Lemma \ref{conv_ground_low}.
\end{proof}
\begin{proof}[Proof of Theorem \ref{sub_class_thresh}] It is a direct consequence of Lemma \ref{conv_ground_high}, in the case $\mathcal{MK}(u_0,v_0)>1$, of the variational characterization of the ground states in the case $\mathcal{MK}(u_0,v_0)=1$ and of Lemma \ref{conv_ground_low} in the case $\mathcal{MK}(u_0)<1$.
\end{proof}




\addtocontents{toc}{\protect\vspace*{\baselineskip}}






\newcommand{\Addresses}{{
  \bigskip
  \footnotesize

  L. Campos, \textsc{IMECC-UNICAMP, Rua Sérgio Buarque de Holanda, 651, Cidade Universitária, 13083-859, Campinas, São Paulo, Brazil.}\par\nopagebreak
  \textit{E-mail address:} \texttt{luccasccampos@gmail.com}
  
  \vspace{3mm}
   A. Pastor, \textsc{IMECC-UNICAMP, Rua Sérgio Buarque de Holanda, 651, Cidade Universitária, 13083-859, Campinas, São Paulo, Brazil.}\par\nopagebreak
  \textit{E-mail address:} \texttt{apastor@ime.unicamp.br}

}}
\setlength{\parskip}{0pt}
\Addresses
\batchmode
\end{document}